\documentclass[reqno,10pt]{amsart}
\usepackage{amssymb}
\usepackage{latexsym}
\usepackage{hyperref}
\usepackage{amsmath}
\usepackage{amscd}
\usepackage{color}
\usepackage{enumerate}
\usepackage{amsfonts}
\usepackage{graphicx}
\usepackage{mathrsfs}
\usepackage{pifont}
\usepackage{setspace}

\usepackage{tabu}
\newcommand{\eps}{{\varepsilon}}
\theoremstyle{plain}
\newtheorem{theorem}{Theorem}
\newtheorem{proposition}[theorem]{Proposition}
\newtheorem{lemma}[theorem]{Lemma}
\newtheorem{corollary}[theorem]{Corollary}

\theoremstyle{definition}
\newtheorem{definition}[theorem]{Definition}
\newtheorem{remark}[theorem]{Remark}

\numberwithin{equation}{section}
\numberwithin{theorem}{section}
\let\Re=\undefined\DeclareMathOperator*{\Re}{Re}
\let\Im=\undefined\DeclareMathOperator*{\Im}{Im}

\def\ge{\geqslant}
\def\le{\leqslant}
\def\geq{\geqslant}
\def\leq{\leqslant}
\def\b{\boldsymbol}
\def\N{\mathbb{N}}
\def\R{\mathbb{R}}
\def\C{\mathbb{C}}

\begin{document}
\title[NLS system in $\mathbb{R}^6$]{Scattering theory For Quadratic Nonlinear Schr\"odinger System in dimension six}

\author[C. Gao]{Chuanwei Gao}
\address{Beijing International Center for Mathematical Research,  Peking University,  \ Beijing, \ China, \ 100871}
\email{cwgao@pku.edu.cn}

\author[F. Meng]{Fanfei Meng}
\address{Graduate School of China Academy of Engineering Physics,  \ Beijing, \ China, \ 100088}
\email{mengfanfei17@gscaep.ac.cn}

\author[C. Xu]{Chengbin Xu}
\address{Graduate School of China Academy of Engineering Physics,  \ Beijing, \ China, \ 100088}
\email{xuchengbin19@gscaep.ac.cn}

\author[J. Zheng]{Jiqiang Zheng}
\address{Institute for Applied Physics and Computational Mathematics,  \ Beijing, \ China, \ 100088}
\email{zhengjiqiang@gmail.com, zheng\_jiqiang@iapcm.ac.cn}

\subjclass[2010]{Primary 35Q55}
\date{\today}
\keywords{Energy-critical, quadratic nonlinear Schr\"odinger system, blow-up, scatter.}
\maketitle

\begin{abstract}
In this paper, we study the solutions to the energy-critical quadratic nonlinear Schr\"odinger system in $ {\dot H}^1 \times {\dot H}^1 $, where the sign of its potential energy can not be determined directly.
If the initial  data ${\rm u}_0$ is radial or non-radial but satisfies the mass-resonance condition, and its energy is below that of the ground state,  using the compactness/rigidity method, we give a complete classification of scattering versus blowing-up dichotomies depending on whether the kinetic energy of  ${\rm u}_0$ is below or above that of the ground state.
\end{abstract}

\maketitle
\section{Introduction}
We consider the quadratic nonlinear Schr\"odinger system:
\begin{equation}\tag{$\text{NLS system}$}
\left\{
\begin{aligned}
& i \partial_{t}\b{\rm u} + A \b{\rm u} + \b{\rm f}(\b{\rm u}) = \b{\rm 0}, \quad (t, x) \in \mathbb{R} \times \mathbb{R}^d,\\
& \b{\rm u}(0,x)= \b{\rm u}_0(x),
\end{aligned}
\right.
\label{NLS system}
\end{equation}
where $ \b{\rm u} $, $ \b{\rm u}_0 $, and $ \b{\rm 0} $ are all vector-valued complex functions with two components  defined as follows
\begin{equation*}
\b{\rm u} := \left(
\begin{aligned}
u \\
v
\end{aligned}
\right),~\b{\rm u_0} := \left(
\begin{aligned}
u_0 \\
v_0
\end{aligned}
\right),~ \b{\rm 0}: = \left(
\begin{aligned}
0 \\
0
\end{aligned}
\right),
\end{equation*}
and $ A $ is a $ 2 \times 2$ matrix, $ \b{\rm f} : \C^2 \to \C^2 $ as follows
\begin{equation}\label{mass-res}
A := \left(
\begin{aligned}
\Delta \quad & ~~~~ 0 \\
0 \quad & ~~~~ \kappa\Delta
\end{aligned}
\right), ~ \b{\rm f}(\b{\rm u}):= \left(
\begin{aligned}
v\overline{u} \\
u^2
\end{aligned}
\right),
\end{equation}
$ u,v: \mathbb{R} \times \mathbb{R}^{d} \rightarrow \mathbb{C} $ are unknown functions, $ \kappa \in (0, \infty) $ is a real number and $ \Delta $ denotes the Laplacian in $ \mathbb{R}^{d} $.

From a  physical point view, \eqref{NLS system} is related to the Raman amplification in a plasma. This process is a nonlinear instability phenomenon (see \cite{Colin2009, Colin2016, Kivshar2000, Koynov1998} for more details). Solutions to \eqref{NLS system} preserve the \emph{mass}, \emph{energy} and \emph{momentum}, defined respectively by
\begin{equation*}
\begin{aligned}
M(\b{\rm u}) :& = \Vert u \Vert_{L^2}^2 + \Vert v \Vert_{L^2}^2 \equiv M(\b{\rm u}_{0}), \\
E(\b{\rm u}) :& = H(\b{\rm u}) - R(\b{\rm u}) \equiv E(\b{\rm u}_{0}), \\
P(\b{\rm u}) :& = \Im \int_{\mathbb{R}^d} \left( \overline{u}\nabla u + \frac{1}{2}\overline{v}\nabla v \right) {\rm d}x \equiv P(\b{\rm u}_0),
\end{aligned}
\end{equation*}
where
\[
\begin{aligned}
  & (\text{kinetic energy}) \quad & H(\b{\rm u}): & = \Vert u \Vert_{{\dot H}^1}^2 + \frac{\kappa}{2} \Vert v \Vert_{{\dot H}^1}^2, \\
  & (\text{potential energy}) \quad & R(\b{\rm u}) :& = \Re \int_{\mathbb{R}^d} \overline{v}u^2 {\rm d}x.
\end{aligned}
\]

The equation \eqref{NLS system} is invariant under the scaling
\[
\b{\rm u}_{\lambda}(t, x) = \lambda^2 \b{\rm u} \left( \lambda^2 t, \lambda x \right)
\]
for any $ \lambda > 0 $.
The critical regularity of Sobolev space is $ {\rm \dot{H}}^{s_c} $,
i.e. $ \Vert \b{\rm u}_{\lambda} \Vert_{{\rm \dot{H}}_{x}^{s_c}(\mathbb{R}^d)} =  \Vert \b{\rm u} \Vert_{{\rm \dot{H}}_{x}^{s_c}(\mathbb{R}^d)} $ where $ s_c = \frac{d}{2} - 2 $.
Therefore, the equation \eqref{NLS system} is called mass-subcritical if $ d \le 3 $, mass-critical if $ d=4 $, energy-subcritical if $ d=5 $, and energy-critical if $ d=6 $. We say \eqref{NLS system} satisfies the mass-resonance condition if $\kappa=\frac{1}{2}$. A distinguished and exclusive feature of this case is that  \eqref{NLS system} is left invariant under the following Galilean transformation
\[
\left(
\begin{aligned}
u(t, x) \\
v(t, x)
\end{aligned}
\right) \to \left(
\begin{aligned}
e^{ix\cdot\xi}e^{-it\vert\xi\vert^2}u(t, x-2t\xi) \\
e^{2ix\cdot\xi}e^{-2it\vert\xi\vert^2}v(t, x-2t\xi)
\end{aligned}
\right)
\]
for any $ \xi \in \mathbb{R}^d $.

From a mathematical point view,  even though the fact $ A \le 0 $ (i.e. $ \langle A\b{\rm u}, \b{\rm u} \rangle \le 0, ~ \forall ~ \b{\rm u} \in {\rm\dot H}^1(\R^d) $) superficially suggests that \eqref{NLS system} is similar to the classical nonlinear Schr\"odinger equation:
\[\tag{$\text{NLS}$}
i\partial_{t}u + \Delta u + \mu \vert u \vert^{p-1} u = 0, \quad (t, x) \in \mathbb{R} \times \mathbb{R}^d,
\label{NLS}
\]
the truth is quite different. One prominent difference among others is that there is no classification of the so-called focusing ($ \mu = 1 $) and defocusing ($ \mu = -1 $) in \eqref{NLS system}, which can be seen from the fact the sign of potential energy $ R(\b{\rm u}) $ can not be judged directly from $\mu$.

In this paper, we consider the energy-critical case of \eqref{NLS system} in  the critical Sobolev space, that is, $ d = 6 $ and $ \b{\rm u}_0 \in {\rm\dot H}^1(\R^6) $.
For any $ s \ge 0 $, we use $ {\dot{\rm H}}^s $ to denote Hilbert space $ {\dot H}^s \times {\dot H}^s $, and $ {\rm H}^s $ to denote Hilbert space $ H^s \times H^s $.
For any $ p \ge 1 $, we use $ {\rm L}^p $ to denote $ L^p \times L^p $.
We also denote the Schr\"odinger group $ \verb"S"(t) $ as follows
\[
\verb"S"(t) := \left(
\begin{aligned}
e^{it\Delta} \qquad 0 \quad \\
\quad 0 \qquad e^{\kappa it\Delta}
\end{aligned}
\right).
\]
By solution, we mean a function $ \b{\rm u} \in C_{t}(I, {\rm\dot H}_{x}^1(\mathbb{R}^6)) \cap {\rm L}_{t, x}^4(I \times \mathbb{R}^6) $ on an interval $ I \ni 0 $ satisfying the Duhamel formula
\[
\b{\rm u}(t) = \verb"S"(t) \b{\rm u}_0 + i \int_{0}^{t} \verb"S"(t - \tau) \b{\rm f}(\b{\rm u}(\tau)) {\rm d}\tau
\label{def}
\]
for $ t \in I $.
If the maximal interval of existence $ I_{\max} = (T_-(\b{\rm u}), T_+(\b{\rm u})) = \mathbb{R} $, we call the solution $ \b{\rm u} $ is global.

We say that a solution $ \b{\rm u} $ to \eqref{NLS system} blows up forward in time if there exists a time $ t_1 \in I_{\max} = (T_-(\b{\rm u}), T_+(\b{\rm u})) $ such that
\[
S_{[t_1, T_+(\b{\rm u}))} (\b{\rm u}) = \infty,
\]
that $ \b{\rm u} $ blows up backward in time if there exists a time $ t_1 \in I_{\max} $ such that
\[
S_{(T_-(\b{\rm u}), t_1]} (\b{\rm u}) = \infty,
\]
and that a global solution $ \b{\rm u} $ scatters if there exists $ \b{\rm u}_{\pm} \in {\rm\dot H}^1(\mathbb{R}^6) $ such that
\begin{equation*}
\lim_{t \to \pm \infty} \left\Vert \b{\rm u}(t) - \verb"S"(t) \b{\rm u}_{\pm} \right\Vert_{{\rm\dot H}^1(\mathbb{R}^6)} = 0.
\end{equation*}

\begin{definition}[Trivial scattering solution, \cite{Masaki2021}]\label{def:trialsc}
We say a solution to a nonlinear equation is a \emph{trivial scattering solution} if it also solves the corresponding linear equation.
(In other words, the  equation does not have nonlinear interaction.)
\end{definition}

\begin{remark}
For \eqref{NLS}, zero is the only trivial scattering solution.
While for \eqref{NLS system}, $ (u, v) = (0, e^{\kappa it\Delta} v_0) $ is  a trivial scattering solution for any $ v_0 \in {\dot H}^1(\R^6) $.
\end{remark}

\begin{definition}[Scattering size, \cite{Visan2012}]\label{scattering size}
The \emph{scattering size} of a solution $ \b{\rm u} = (u, v)^T $ to \eqref{NLS system} on the interval $ I \ni 0 $ is
\[
S_I (\b{\rm u}) = \Vert \b{\rm u} \Vert_{{\rm L}_{t, x}^4(I \times \R^6)}^4 \sim \int_I \int_{\R^6} \vert u(t, x) \vert^4 + \vert v(t, x) \vert^4 {\rm d}x {\rm d}t.
\]
\end{definition}

\begin{remark}[$ {\rm\dot H}^1 $ sacttering]\label{sc}
It is easy to verify that the solution to \eqref{NLS system} will scatter in $ {\rm\dot H}^1(\mathbb{R}^6) $ as $ t \to + \infty $ if the solution $ \b{\rm u}(t) $ to \eqref{NLS system} is global in $ {\rm\dot H}^1(\mathbb{R}^6) $ with finite scattering size on $ \R $, i.e., $ S_{\R} (\b{\rm u}) < + \infty $.
\end{remark}

The scattering solutions to \eqref{NLS} are attractive to many  researchers,  such as \cite{Dodson2016, Duyckaerts2010, Holmer2008, Miao2013} for the focusing case and \cite{Miao2014, Visan2007} for the defocusing case and so on. A quite interesting problem is the defoucusing energy-critical case, i.e. $p=1+\frac{4}{d-2}$, one may refer to \cite{Bourgain1999, Colliander2006, Killip2010, Tao2005} for more details.

For \eqref{NLS system} and its related models, there are amount of literature on those themes, such as \cite{Pastor2015, Pastor2019, Uriya2016} for three waves interaction and \cite{Hayashi2013, Inui2018, Noguera20201} for quadratic-type interaction.
In \cite{Hayashi2011, Iwabuchi2016, Ogawa2015}, the authors have studied the problems of \eqref{NLS system} that arise in 2D.
There are also many papers  on  the well-posedness and dynamics of the solutions to \eqref{NLS system} with data below the ground state $ \b{\rm Q} $ in  dimensions, $ 1 \le d \le 5 $, where $ \b{\rm Q} := (\phi, \varphi)^T \in D_d $ is a pair of non-negative real-valued functions satisfying:
\begin{equation}
\left\{
\begin{aligned}
& \phi - \Delta \phi = \phi \varphi, \\
& 2\varphi - \kappa\Delta \varphi = \phi^2,
\end{aligned}
\right.
\quad x \in \R^d,
\label{d}
\end{equation}
with $ D_d := \left\{ \b{\rm G} \in {\rm H}^1(\R^d) ~ \big\vert ~ R(\b{\rm G}) > 0 \right\} $.  Among others, up to translations, Hayashi, Li and  Naumkin \cite{Hayashi2011} proved the uniqueness of $ \b{\rm Q}:=(\phi_1, \frac{\phi_1}{\sqrt2})^T $ when $ \kappa = 2 $, where $ \phi_1 \ne 0 $ is the unique non-negative radial real-valued solution to
\[
\phi_1 - \Delta \phi_1 = \frac{1}{\sqrt2} \phi_1^2.
\]
With and without mass-resonance, the mass-critical case ($ d = 4 $) was studied in \cite{Inui2019} and the energy-subcritical case ($ d = 5 $) was studied separately in \cite{Hamano2018, Hamano2019, Meng2021, Noguera2020, Wang2019}.

In this paper, we concern about the energy-critical case ($ d = 6 $) of \eqref{NLS system}.
Our results show the classification of the solution as to scattering or blowing-up depending on the relationship between the initial data $u_0$ and the ground state $ \b{\rm W} := (\phi, \varphi)^T \in D_6 $, where $(\phi, \varphi)^t$ is a pair of non-negative real-valued functions satisfying:
\begin{equation}
\left\{
\begin{aligned}
& - \Delta \phi = \phi \varphi, \\
& - \kappa\Delta \varphi = \phi^2,
\end{aligned}
\right.
\quad x \in \mathbb{R}^6,
\label{gs}
\end{equation}
and $ D_6 := \left\{ \b{\rm G} \in {\rm\dot H}^1(\R^6) \cap {\rm L}^3(\R^6) ~ \big\vert ~ R(\b{\rm G}) > 0 \right\} $.

At this point, some useful information of $W$ are discussed. Define the operator $ T : ({\dot H}^1 \cap L^3)_r(\R^6) \to H^1(\R) $ as follows:
\[(Tf)(t) := e^{2t} f(e^t), \qquad t \in \R.
\]
By showing the map $T$ gives a one-to-one correspondence between the ground states of \eqref{d} and those of \eqref{gs}, up to translations and dilations, \cite{Hayashi2011}  proved the uniqueness of $ \b{\rm W} $ for any $ \kappa \in (0, \infty) $.  Therefore, $ \b{\rm W} = (\phi_0, \frac{\phi_0}{\sqrt{\kappa}})^T $,  where $ \phi_0 \ne 0 $ is the unique non-negative radial real-valued solution to
\begin{equation}
- \Delta \phi_0 = \frac{1}{\sqrt{\kappa}} \phi_0^2.
\label{6}
\end{equation}
Furthermore, we may write $\phi_0$ explicitly as follows
\[
\phi_0(x) = \frac{\sqrt{\kappa}}{\left( 1 + \frac{\vert x \vert^2}{24} \right)^2} \in {\dot H}^1(\R^6),
\]
one may refer to \cite{Kenig2006} for more details.

 Our main results read as follows:
\begin{theorem}
\label{classification}
For $ d = 6 $, $ \b{\rm u}_0 \in {\rm\dot H}^1(\R^6) $, and $ E(\b{\rm u_0}) < E(\b{\rm W}) $ in the \eqref{NLS system},
\begin{enumerate}
\item if $ H(\b{\rm u_0}) < H(\b{\rm W}) $, then $ \b{\rm u} $ scatters for $ \b{\rm u}_0 $ is non-radial with the mass-resonance condition or $ \b{\rm u}_0 $ is radial;
\item if $ H(\b{\rm u_0}) > H(\b{\rm W}) $, then $ \b{\rm u} $ blows up in finite time for $u_0$ is radial with $ x \b{\rm u}_0 \in {\rm L}^2(\R^6) $ or $ \b{\rm u}_0 \in {\rm H}^1(\R^6) $.
\end{enumerate}
\end{theorem}

\begin{remark}
From the proof of Proposition \ref{coer}, the restriction $ E(\b{\rm u_0}) < E(\b{\rm W}) $ implies $ H(\b{\rm u_0}) \ne H(\b{\rm W}) $,  which indicates the completeness of our classification.

\end{remark}

\begin{remark}[Reason of the mass-resonance condition]
The mass-resonance condition enables us to use the Galilean invariance of the critical solution $ \b{\rm u}_c $ to \eqref{NLS system} which is very useful in proving zero momentum  $ P(\b{\rm u}_c(t)) \equiv \b{\rm 0} $ for the non-radial case and is only available when $\kappa=\frac{1}{2}$, that is the essential reason  why we require  $ \kappa = \frac12 $ for the non-radial case in Theorem \ref{classification}. One can refer to the proof of Proposition \ref{P=0} in this paper or the statement of Remark 1.2 in \cite{Meng2021} for more details. Under the assumption $ \b{\rm u}_0 $ is non-radial with the mass-resonance condition or $ \b{\rm u}_0 $ is radial, the scattering part of the results in Theorem \ref{classification} is consistent with that of the mass-critical case ($ d = 4 $) and the energy-subcritical case ($ d = 5 $). One can refer to \cite{Hamano2018, Hamano2019, Inui2019, Meng2021, Noguera2020, Wang2019}.
\end{remark}

 To prove the blowing-up result in Theorem \ref{classification}, we exploit the Virial identity in Subsection \ref{idea}.
To prove the scattering result, we follow the strategies  from \cite{Bourgain1999, Kenig2006, Tao2005} through analysing a kind of special solution, so called ``critical solution", in Section \ref{nonradial} and Section \ref{non}.
After a series of reductions, the ``critical solution"  shares the property of almost periodicity modulo symmetries:

\begin{definition}[Almost periodic solution, \cite{Visan2012}]\label{almost periodic}
We call a solution $ \b{\rm u} = (u, v)^T $ to \eqref{NLS system} is \emph{almost periodic modulo symmetries} with lifespan $ I \ni 0 $, if there exist functions $ \lambda : I \to \R^+ $, $ x : I \to \R^6 $, and $ C : \R^+ \to \R^+ $ such that
\[
\int_{\vert x - x(t) \vert \ge \frac{C(\eta)}{\lambda(t)}} \left( \vert \nabla u(t, x) \vert^2 + \frac{\kappa}{2} \vert \nabla v(t, x) \vert^2 \right) {\rm d}x \le \eta
\]
and
\[
\int_{\vert \xi \vert \ge C(\eta)\lambda(t)} \left( \vert \xi \vert^2 \vert \hat{u}(t, \xi) \vert^2 + \frac{\kappa}{2} \vert \xi \vert^2 \vert \hat{v}(t, \xi) \vert^2 \right) {\rm d}\xi \le \eta
\]
for all $ t \in I $ and $ \eta > 0 $.
We refer to the function $ \lambda $ as the \emph{frequency scale function}, $ x $  as the \emph{spatial center function}, and $ C $ as the \emph{compactness modulus function}.
\end{definition}

 From the Arzela-Ascoli Theorem, we have

\begin{lemma}[Compactness in $ L^2 $]\label{compactness in L^2}
The set $ \mathcal{F} \subset L^2(\R^d) $ is compact in $ L^2(\R^d) $ if and only if the following conditions hold:

(i) there exists $ A > 0 $ such that $ \Vert f \Vert_{L^2(\R^d)} \le A, \quad \forall ~ f \in \mathcal{F} $;

(ii) for any $ \varepsilon > 0 $, there exists $ R = R(\varepsilon) > 0 $ such that
\[
\int_{\vert x \vert \ge R} \vert f(x) \vert^2 {\rm d}x + \int_{\vert \xi \vert \ge R} \vert \hat{f}(\xi) \vert^2 {\rm d}\xi < \varepsilon, \qquad \forall ~ f \in \mathcal{F}.
\]
\end{lemma}

\begin{remark}\label{compactness in H^1}
By Lemma \ref{compactness in L^2}, the set $ \mathcal{F}_1 \subset {\dot H}^1(\R^d) $ is compact in $ {\dot H}^1(\R^d) $ if and only if $ \mathcal{F}_1 $ is bounded  in $\dot{H}^1$ uniformly and for any $ \eta > 0 $, there exists a compactness modulus function $ C(\eta) > 0 $ such that
\begin{equation*}
\int_{\vert x \vert \ge C(\eta)} \vert \nabla f(x) \vert^2 {\rm d}x + \int_{\vert \xi \vert \ge C(\eta)} \vert \xi \vert^2 \vert \hat{f}(\xi) \vert^2 {\rm d}\xi < \eta, \qquad \forall ~ f \in \mathcal{F}_1.
\end{equation*}
\end{remark}

By using the Palais-Smale condition modulo symmetries, Proposition \ref{PS}, we know that the critical solutions have to be almost periodic and minimal-kinetic-energy blowing-up which is proven in Subsection \ref{birth} and stated as follows:
\begin{theorem}\label{minH}
If the scattering result of Theorem \ref{classification} fails, then there exists a solution $ \b{\rm u}_c : I_c \times \R^6 \to \C^2 $ with the maximal interval of existence $ I_c $ satisfying
\[
\sup_{t \in I_c} H(\b{\rm u}_c(t)) < H(\b{\rm W}),
\]
which is almost periodic modulo symmetries and blows up in both directions. Furthermore, $ \b{\rm u}_c $ has the minimal kinetic energy among all the blowing-up solutions, that is,
\[
\sup_{t \in I_c} H(\b{\rm u}_c(t)) \le \sup_{t \in I} H(\b{\rm u}(t)),
\]
where $ \b{\rm u} : I \times \R^6 \to \C^2 $ is a solution to \eqref{NLS system} with the maximal interval of existence $ I $ and blows up in at least one direction.
\end{theorem}

For the blow-up solutions $ \b{\rm u}_c $ to \eqref{NLS system} in Theorem \ref{minH}, we can classify them by their different properties of the frequency scale function $ \lambda(t) : I_c \to \R^+ $.
\begin{proposition}[Classification of critical solution, \cite{Killip2010, Killip2013}]\label{enemies}
If the scattering result of Theorem \ref{classification} fails, then there exists a minimal-kinetic-energy blowing-up solution $ \b{\rm u}_c : I_c \times \R^6 \to \C^2 $ which is almost periodic modulo symmetries with the maximal interval of existence $ I_c $ satisfying
\begin{equation}
S_{I_c} (\b{\rm u_c}) = \infty, \quad \text{and} \quad
\sup_{t \in I_c} H(\b{\rm u}_c(t)) < H(\b{\rm W}).
\label{u_c}
\end{equation}
Furthermore, $ \b{\rm u}_c $ has to be one of the following three cases:
\begin{enumerate}
\item a \emph{blowing-up solution in finite time} if
 \[
 \vert \inf I_c \vert < \infty, \quad \text{or} \quad \sup I_c < \infty;
 \]
\item a \emph{soliton-like solution} if
 \[
 I_c = \R, \quad \text{and} \quad \lambda(t) = 1, ~ \forall ~ t \in \R;
 \]
\item a \emph{low-to-high frequency cascade} if
 \[
 I_c = \R, \quad \text{and} \quad \inf_{t \in \R} \lambda(t) = 1, \quad \limsup_{t \to +\infty} \lambda(t) = \infty.
 \]
\end{enumerate}
\end{proposition}

After clarifying our three enemies, we need to rule out them according to their respective properties. We will prove that each critical solution has to be $ \b{\rm 0} $ and eliminate the possibility of their existence one by one. For the first enemy, blowing-up solution in finite time, we analysis its compactness and use the Virial identity to obtain $ \b{\rm u}_c \equiv \b{\rm 0} $ in Section \ref{nonradial}.

For the second and third enemies, by using the reduced Duhamel formulas(see Proposition \ref{reduced} below), we will show that they enjoy the following negative regularity in Subsection \ref{properties}.
\begin{theorem}[Negative regularity]\label{negative regularity}
Let $ \b{\rm u}_c $ be as in Proposition \ref{enemies} with $ I_c = \R $.
If
\[
\sup_{t \in \R} H(\b{\rm u}_c(t)) < \infty \quad \text{and} \quad \inf_{t \in \R} \lambda(t) \ge 1,
\]
then there exists $ \eps > 0 $ such that $ \b{\rm u}_c \in {\rm L}_t^\infty(\R, {\rm\dot H}^{-\eps}(\R^6)) $.
\end{theorem}

It is easy to see that the momentum of any radial function is always zero.  For the non-radial case, with additional mass-resonance condition, the above fact can also be recovered. This  observation enables us to deduce $ P(\b{\rm u}_c) \equiv \b{\rm 0} $, if  $ \b{\rm u}_{c, 0} $ is non-radial with the mass-resonance condition or $ \b{\rm u}_{c, 0} $ is radial, which is shown  in Proposition \ref{P=0}. The zero momentum property helps us to obtain the control of spatial center function $ x(t) $ of the soliton-like solution in  Corollary \ref{x(t)}.

 By Theorem \ref{negative regularity} and Corollary \ref{x(t)}, we use the Virial identity to eliminate the possibility of soliton-like solution.
Finally, in order to exclude the last enemy, we divide its frequencies into low and high parts, and conquer them respectively. To be more precise, for the low frequency, we will use the negative regularity,  and for the high frequency, we will use  the property of low-to-high frequency cascade.

\subsection{Outline of the paper}
The following part of this paper is organized as follows: In Section \ref{PRE}, we clarify some preliminaries, including Bernstein inequalities, Strichartz estimate, local well-posedness, and long time perturbation theory of \eqref{NLS system}.
Section \ref{var} are divided into three parts.
We focus on ground state in Subsection \ref{ground}, show the energy trapping in Subsection \ref{et}, and give the proof of blowing-up result of Theorem \ref{classification} in Subsection \ref{idea}.
Section \ref{LPD} concentrates on the proof of Theorem \ref{minH} by the properties of the linear profile decomposition and the nonlinear profile.
To obtain the scattering result, we rule out our three enemies arisen in Proposition \ref{enemies} one by one in Section \ref{nonradial} and Section \ref{non}.
Combining with the result of blowing-up in finite time in Theorem \ref{vs} and Theorem \ref{bu}, we complete the proof of Theorem \ref{classification} finally.

\section{Preliminaries}\label{PRE}
We mark $ A \lesssim B $ to mean there exists a constant $ C > 0 $ such that $ A \leqslant C B $.
We indicate dependence on parameters via subscripts, e.g. $ A \lesssim_{x} B $ indicates $ A \leqslant CB $ for some $ C = C(x) > 0 $.
We use $ A \sim B $ to denote $ A \lesssim B \lesssim A $.

Let $ \phi(\xi) $ be a radial smooth function supported in the ball $ \{ \xi \in \R^6 : \vert \xi \vert \le \frac{11}{10} \} $ and equal to 1 on the ball $ \{ \xi \in \R^6 : \vert \xi \vert \le 1 \} $.
For each number $ N > 0 $, we define the Fourier multipliers
\begin{equation*}
\begin{aligned}
\widehat{P_{\le N} \b{\rm g}} (\xi) & ~ := \phi \left( \frac{\xi}{N} \right) \hat{\b{\rm g}}(\xi), \\
\widehat{P_{> N} \b{\rm g}} (\xi) & ~ := \left( 1 - \phi \left( \frac{\xi}{N} \right) \right) \hat{\b{\rm g}}(\xi), \\
\widehat{P_N \b{\rm g}} (\xi) & ~ := \left( \phi \left( \frac{\xi}{N} \right) - \phi \left( \frac{2\xi}{N} \right) \right) \hat{\b{\rm g}}(\xi),
\end{aligned}
\end{equation*}
and similarly $ P_{< N} $ and $ P_{\ge N} $. We also define
\[
P_{M < \cdot \le N} := P_{\le N} - P_{\le M} = \sum_{M < N^\prime \le N} P_{N^\prime}
\]
whenever $ M < N $.
We will usually use these multipliers when $ M $ and $ N $ are \emph{dyadic numbers} (that is, of the form $ 2^n $ for some integer $ n $);
in particular, all summations over $ N $ or $ M $ are understood to be over dyadic numbers.
Nevertheless, it will occasionally be convenient to allow $ M $ and $ N $ to not be a power of 2.
Like all Fourier multipliers, the Littlewood-Paley operators commute with the propagator $ \verb"S"(t) $, as well as with differential operators such as $ i \partial_t + A $.

\begin{lemma}[Bernstein inequalities]\label{Bern}
For any $ 1 \le p \le q \le \infty $ and $ s \ge 0 $,
\begin{equation*}
\begin{aligned}
\Vert P_{\ge N} \b{\rm g} \Vert_{{\rm L}_x^p(\R^6)} \lesssim_s & ~ N^{-s} \Vert \vert \nabla \vert^s P_{\ge N} \b{\rm g} \Vert_{{\rm L}_x^p(\R^6)} \lesssim_s N^{-s} \Vert \Vert \nabla \vert^s \b{\rm g} \Vert_{{\rm L}_x^p(\R^6)}, \\
\Vert \vert \nabla \vert^s P_{\le N} \b{\rm g} \Vert_{{\rm L}_x^p(\R^6)} \lesssim_s & ~ N^s \Vert P_{\le N} \b{\rm g} \Vert_{{\rm L}_x^p(\R^6)} \lesssim_s N^s \Vert \b{\rm g} \Vert_{{\rm L}_x^p(\R^6)}, \\
\Vert \vert \nabla \vert^{\pm s} P_N \b{\rm g} \Vert_{{\rm L}_x^p(\R^6)} \lesssim_s & ~ N^{\pm s} \Vert P_N \b{\rm g} \Vert_{{\rm L}_x^p(\R^6)} \lesssim_s N^{\pm s} \Vert \b{\rm g} \Vert_{{\rm L}_x^p(\R^6)} ,\\
\Vert P_N \b{\rm g} \Vert_{{\rm L}_x^q(\R^6)} \lesssim & ~ N^{\frac{6}{p} - \frac{6}{q}} \Vert P_N \b{\rm g} \Vert_{{\rm L}_x^p(\R^6)},
\end{aligned}
\end{equation*}
where $ \vert \nabla \vert^s $ is the classical fractional-order operator.
\end{lemma}

\begin{lemma}[Gronwall-type inequality, \cite{Killip2010}]\label{Gronwall}
  Fix $\gamma>0$. Given $0<\eta<(1-2^{-\gamma})/2$ and $\{b_k\}\in l^\infty(\mathbb{N})$, let $\{x_k\}\in l^\infty(\mathbb{N})$ be a non-negative sequence obeying
  \begin{align}
    x_k\leq b_k+\eta \sum_{l=0}^{\infty}2^{-\gamma|k-l|}x_l\  \ \ \ \ \ for\ \ all\ \ k\geq0.
  \end{align}
  Then
  \begin{align}
    x_k\lesssim \sum_{l=0}^{\infty}r^{|k-l|}b_l
  \end{align}
  for some $r=r(\eta)\in(2^{-\gamma},1)$. Moreover, $r\downarrow 2^{-\gamma}$ as $\eta\downarrow0.$
\end{lemma}

\subsection{Strichartz estimate}\label{se}
 Let $\Lambda_s$ be the set of pairs $(p,q)$ with $q \geqslant 2$ and satisfying
\begin{equation}
\frac{2}{q} = 6 \left( \frac{1}{2} - \frac{1}{r} \right) - s.
\label{Lambda}
\end{equation}
Define
\begin{equation*}
\Vert \b{\rm u} \Vert_{\mathcal{S}({\rm\dot H}^s(\mathbb{R}^6))}:= \sup_{(q, r) \in \Lambda_{s}} \Vert \b{\rm u} \Vert_{{\rm L}_t^q(\mathbb{R}, {\rm L}_x^r(\mathbb{R}^6))},
\end{equation*}
and
\begin{equation*}
\Vert \b{\rm u} \Vert_{\mathcal{S}^\prime({\rm\dot H}^s(\mathbb{R}^6))}:= \inf_{(q, r) \in \Lambda_{s}} \Vert \b{\rm u} \Vert_{{\rm L}_t^{q^\prime}(\mathbb{R}, {\rm L}_x^{r^\prime}(\mathbb{R}^6))}.
\end{equation*}
We extent our notation $ \mathcal{S}({\rm\dot H}^s(\mathbb{R}^6)), \mathcal{S}^\prime({\rm\dot H}^s(\mathbb{R}^6)) $ as follows: we write $ \mathcal{S}(I, {\rm\dot H}^s(\mathbb{R}^6)) $ or $ \mathcal{S}^\prime(I, {\rm\dot H}^s(\mathbb{R}^6)) $ to indicate a restriction to a time subinterval $ I \subset \mathbb{R} $.

\begin{lemma}[Dispersive estimate]\label{disper}
For any $ t \ne 0 $ and $ r \ge 2 $, we have
\begin{equation*}
\begin{aligned}
\Vert \verb"S"(t) \b{\rm g} \Vert_{{\rm L}_x^\infty(\R^6)} \lesssim_\kappa & ~  \vert t \vert^{-3} \Vert \b{\rm g} \Vert_{{\rm L}^1(\R^6)}, \\
\Vert \verb"S"(t) \b{\rm g} \Vert_{{\rm L}_x^r(\R^6)} \lesssim_\kappa & ~  \vert t \vert^{\frac{6}{r} - 3} \Vert \b{\rm g} \Vert_{{\rm L}^{r^\prime}(\R^6)},
\end{aligned}
\end{equation*}
where $ r^\prime $ is the H\"older conjugation index, that is, $ \frac{1}{r} + \frac{1}{r^\prime} = 1 $.
\end{lemma}

By combing the endpoint results in \cite{Keel1998} with Lemma \ref{disper}, we can use $ T-T^\star $ method to obtain the following Strichartz estimates.

\begin{theorem}[Strichartz estimates, \cite{Meng2021}]\label{strichartz}
The solution $ \b{\rm u} $ to \eqref{NLS system} on an interval $ I \ni t_{0} $ obeys
\begin{equation}
\Vert \b{\rm u} \Vert_{\mathcal{S}(I, {\rm L}^2(\mathbb{R}^6))} \le C \left( \Vert \b{\rm u}(t_{0}) \Vert_{{\rm L}^2(\mathbb{R}^6)} + \Vert \b{\rm f}(\b{\rm u}) \Vert_{\mathcal{S}^\prime(I, {\rm L}^2(\mathbb{R}^6))} \right).
\label{Strichartz}
\end{equation}
\end{theorem}

\subsection{Local well-posedness}\label{lwp}
 Standard argument can help us prove the following basic local well-posedness properties of the solution to \eqref{NLS system}.
\begin{theorem}[Local well-posedness, \cite{Cazenave2003, Miao2014}]\label{sd}
For $ \b{\rm u}_0 \in {\rm\dot H}^1 (\mathbb{R}^6) $, there exists a unique maximal-lifespan solution $ \b{\rm u} : I \times \R^6 \to \C^2 $ with the following properties holding
\begin{enumerate}
\item $ 0 \in I $ is an open interval.
\item If $ \sup I $ is finite, then $ \b{\rm u} $ blows-up forward in time. Similarly, if $ \inf I $ is finite, then $ \b{\rm u} $ blows-up backward in time.
\item If $ \sup I = +\infty $ and $ \b{\rm u} $ does not blow-up forward in time, then $ \b{\rm u} $ scatters forward in time.
Conversely, given $ \b{\rm u}_+ \in {\rm\dot H}^1(\R^6) $, there is a unique solution $ \b{\rm u}(t) $ to \eqref{NLS system} in a neighborhood of $ t = \infty $ such that
\begin{equation*}
\lim_{t \to +\infty} \left\Vert \b{\rm u}(t) - \verb"S"(t) \b{\rm u}_+ \right\Vert_{{\rm\dot H}^1(\mathbb{R}^6)} = 0.
\end{equation*}
Analogous statements hold backward in time.
\item There exists a small number $ \delta_{sd} > 0 $ satisfying that
if $ \| \b{\rm u}_0 \|_{\mathcal{S}(I, {\rm\dot H}^1(\mathbb{R}^6))} \le \delta_{sd} $, then $ \b{\rm u} $ is global and scatters with $ S_{\R}(\b{\rm u}) \lesssim \delta_{sd}^4$.
\end{enumerate}
\end{theorem}

\begin{remark}\label{esd}
The last result of Theorem \ref{lwp} can be improved in the following manner: instead of using $ \| \b{\rm u}_0 \|_{\mathcal{S}(I, {\rm\dot H}^1(\mathbb{R}^6))} \le \delta_{sd} $,
we may use a weaker hypothesis $ \left\Vert \verb"S"(t) \b{\rm u}_0 \right\Vert_{\mathcal{S}(I, {\dot{\rm H}}^1(\mathbb{R}^6))} \le \delta_{sd} $  to obtain the same result.
And for any given $ \b{\rm u}_0 \in {\dot{\rm H}}^1(\mathbb{R}^6) $, one can utilize Strichartz estimate to deduce that  there exists $ I \ni 0 $ such that $ \left\Vert \verb"S"(t) \b{\rm u}_0 \right\Vert_{\mathcal{S}(I, {\dot{\rm H}}^1(\mathbb{R}^6))} \le \delta_{sd}$.
\end{remark}

Inspired by \cite{Tao2005}, we can establish the following stability result under weak smallness conditions.
\begin{proposition}[Long time perturbation theory]\label{pt}
For each $ M \gg 1 $, there exist $ \varepsilon = \varepsilon (M) \ll 1 $ and $ L = L (M) \gg 1 $such that the following holds.
Let $ \b{\rm u} = \b{\rm u}(t, x) \in {\rm\dot H}_x^1(\mathbb{R}^6) $ for all $ t $ and solve
\begin{equation*}
i \b{\rm u}_t + A \b{\rm u} + \b{\rm f}(\b{\rm u}) = \b{\rm 0}.
\end{equation*}
Let $ \tilde{\b{\rm u}} = \tilde{\b{\rm u}}(t, x) \in {\rm\dot H}_x^1(\mathbb{R}^6) $ for all $ t $ and define
\begin{equation*}
\b{\rm e} := i \tilde{\b{\rm u}}_t + A \tilde{\b{\rm u}} + \b{\rm f}(\tilde{\b{\rm u}}).
\end{equation*}
If
\begin{equation*}
\begin{aligned}
& \Vert \tilde{\b{\rm u}} \Vert_{\mathcal{S}({\rm\dot H}^1(\mathbb{R}^6))} \le M, \qquad
\Vert \nabla \b{\rm e} \Vert_{\mathcal{S}^\prime({\rm L}^2(\mathbb{R}^6))} \le \varepsilon, \\
& \Vert \verb"S"(t-t_0) ( \b{\rm u}_0 - \tilde{\b{\rm u}}(t_0) ) \Vert_{\mathcal{S}({\rm\dot H}^1(\mathbb{R}^6))} \le \varepsilon,
\end{aligned}
\end{equation*}
then
\begin{equation*}
\Vert \b{\rm u} \Vert_{\mathcal{S}({\rm\dot H}^1(\mathbb{R}^6))} \le L.
\end{equation*}
\end{proposition}

\begin{proof}
Let $ \b{\rm w} := \b{\rm u} - \tilde{\b{\rm u}} = \left( \begin{aligned} w_1 \\ w_2 \end{aligned} \right) $
for $ \b{\rm u} = \left( \begin{aligned} u \\ v \end{aligned} \right) $
and $ \tilde{\b{\rm u}} = \left( \begin{aligned} \tilde{u} \\ \tilde{v} \end{aligned} \right) $.
Then $ \b{\rm w} $ solves the equation
\begin{equation*}
i \b{\rm w}_t + A \b{\rm w}
+ \left( \begin{aligned} \overline{\tilde{u}}w_2+\tilde{v}\overline{w_1} \\ 2\tilde{u}w_1 \quad \end{aligned} \right)
+ \b{\rm f}(\b{\rm w}) + \b{\rm e} = \b{\rm 0}.
\end{equation*}
Since $ \Vert \tilde{\b{\rm u}} \Vert_{\mathcal{S}({\rm\dot H}^1(\mathbb{R}^6))} \le M $,
we can partition $ [t_0, +\infty) $ into $ N = N(M) $ intervals $ I_j = [t_j, t_{j+1}) $ such that
for each $ 0 \le j \le N-1 $, the quantity $ \Vert \tilde{\b{\rm u}} \Vert_{\mathcal{S}(I_j, {\rm\dot H}^1(\mathbb{R}^6))} \le \delta $ is small enough.
By Duhamel formula,
\begin{equation*}
\b{\rm w}(t) = \verb"S"(t-t_j) \b{\rm w}(t_j) + i \int_{I_j} \verb"S"(t-s) \b{\rm N}(\b{\rm w}(s)) {\rm d}s,
\end{equation*}
where
\begin{equation*}
\b{\rm N}(\b{\rm w}) = \left( \begin{aligned} \overline{\tilde{u}}w_2+\tilde{v}\overline{w_1} \\ 2\tilde{u}w_1 \quad \end{aligned} \right)
+ \b{\rm f}(\b{\rm w}) + \b{\rm e}.
\end{equation*}
By Theorem \ref{strichartz}, we obtain
\begin{equation*}
\Vert \b{\rm w} \Vert_{\mathcal{S}(I_j, {\rm\dot H}^1(\mathbb{R}^6))}
\le \Vert \verb"S"(t-t_j) \b{\rm w}(t_j) \Vert_{\mathcal{S}(I_j, {\rm\dot H}^1(\mathbb{R}^6))}
+ C \Vert \nabla \b{\rm N}(\b{\rm w}) \Vert_{\mathcal{S}^\prime(I_j, {\rm L}^2(\mathbb{R}^6))}.
\end{equation*}
Note that $ \Vert \nabla \b{\rm e} \Vert_{\mathcal{S}^\prime({\rm L}^2(\mathbb{R}^6))} \le \varepsilon $,
\begin{equation*}
\begin{aligned}
\Vert \nabla \b{\rm N}(\b{\rm w}) \Vert_{\mathcal{S}^\prime(I_j, {\rm L}^2(\mathbb{R}^6))}
\le & ~ 3 \Vert \nabla (\tilde{\b{\rm u}} \b{\rm w}) \Vert_{{\rm L}_t^2(I_j, {\rm L}_x^{\frac32}(\mathbb{R}^6))}
+ \Vert \nabla \b{\rm f}(\b{\rm w}) \Vert_{{\rm L}_t^2(I_j, {\rm L}_x^{\frac32}(\mathbb{R}^6))} \\
& ~ + \Vert \nabla \b{\rm e} \Vert_{\mathcal{S}^\prime({\rm L}^2(\mathbb{R}^6))} \\
\le & ~ 3 \Vert \nabla \tilde{\b{\rm u}} \b{\rm w} \Vert_{{\rm L}_t^2(I_j, {\rm L}_x^{\frac32}(\mathbb{R}^6))}
+ 3 \Vert \tilde{\b{\rm u}} \nabla \b{\rm w} \Vert_{{\rm L}_t^2(I_j, {\rm L}_x^{\frac32}(\mathbb{R}^6))} \\
& ~ + \Vert \nabla \b{\rm f}(\b{\rm w}) \Vert_{{\rm L}_t^2(I_j, {\rm L}_x^{\frac32}(\mathbb{R}^6))}
+ \varepsilon \\
\le & ~ 3 \Vert \nabla \tilde{\b{\rm u}} \Vert_{{\rm L}_t^4(I_j, {\rm L}_x^{\frac{12}5}(\mathbb{R}^6))} \Vert\b{\rm w} \Vert_{{\rm L}_t^4(I_j, {\rm L}_x^4(\mathbb{R}^6))} \\
& ~ + 3 \Vert \tilde{\b{\rm u}} \Vert_{{\rm L}_t^4(I_j, {\rm L}_x^4(\mathbb{R}^6))} \Vert \nabla \b{\rm w} \Vert_{{\rm L}_t^4(I_j, {\rm L}_x^{\frac{12}5}(\mathbb{R}^6))} \\
& ~ + 2 \Vert \tilde{\b{\rm w}} \Vert_{{\rm L}_t^4(I_j, {\rm L}_x^4(\mathbb{R}^6))} \Vert \nabla \b{\rm w} \Vert_{{\rm L}_t^4(I_j, {\rm L}_x^{\frac{12}5}(\mathbb{R}^6))}
+ \varepsilon \\
\le & ~ 6 \delta \Vert \b{\rm w} \Vert_{\mathcal{S}(I_j, {\rm\dot H}^1(\mathbb{R}^6))}
+ 2 \Vert \b{\rm w} \Vert_{\mathcal{S}(I_j, {\rm\dot H}^1(\mathbb{R}^6))}^2 + \varepsilon.
\end{aligned}
\end{equation*}
Thus
\begin{equation}
\begin{aligned}
\Vert \b{\rm w} \Vert_{\mathcal{S}(I_j, {\rm\dot H}^1(\mathbb{R}^6))}
\le & ~ \Vert \verb"S"(t-t_j) \b{\rm w}(t_j) \Vert_{\mathcal{S}(I_j, {\rm\dot H}^1(\mathbb{R}^6))} \\
& ~ + 6C \delta \Vert \b{\rm w} \Vert_{\mathcal{S}(I_j, {\rm\dot H}^1(\mathbb{R}^6))}
+ 2C \Vert \b{\rm w} \Vert_{\mathcal{S}(I_j, {\rm\dot H}^1(\mathbb{R}^6))}^2 + C \varepsilon.
\end{aligned}
\label{p1}
\end{equation}

Provided $ \delta < \frac{1}{12C} $ and $ \Vert \verb"S"(t-t_j) \b{\rm w}(t_j) \Vert_{\mathcal{S}(I_j, {\rm\dot H}^1(\mathbb{R}^6))} + C \varepsilon \le \frac{1}{8C} - \frac{3}{2} \delta $, we can get
\begin{equation*}
\Vert \b{\rm w} \Vert_{\mathcal{S}(I_j, {\rm\dot H}^1(\mathbb{R}^6))}
\le 6C \delta \Vert \b{\rm w} \Vert_{\mathcal{S}(I_j, {\rm\dot H}^1(\mathbb{R}^6))}
+ 2C \Vert \b{\rm w} \Vert_{\mathcal{S}(I_j, {\rm\dot H}^1(\mathbb{R}^6))}^2 + \frac{1}{8C} - \frac{3}{2} \delta ,
\end{equation*}
which implies $ \Vert \b{\rm w} \Vert_{\mathcal{S}(I_j, {\rm\dot H}^1(\mathbb{R}^6))} \le \frac{1}{4C} - 3 \delta $.
Under this circumstance, we have
\begin{equation*}
6C \delta \Vert \b{\rm w} \Vert_{\mathcal{S}(I_j, {\rm\dot H}^1(\mathbb{R}^6))}
+ 2C \Vert \b{\rm w} \Vert_{\mathcal{S}(I_j, {\rm\dot H}^1(\mathbb{R}^6))}^2
\le \frac{1}{2} \Vert \b{\rm w} \Vert_{\mathcal{S}(I_j, {\rm\dot H}^1(\mathbb{R}^6))}.
\end{equation*}
Thus,
\begin{equation*}
\Vert \b{\rm w} \Vert_{\mathcal{S}(I_j, {\rm\dot H}^1(\mathbb{R}^6))}
\le \frac{1}{2} \Vert \b{\rm w} \Vert_{\mathcal{S}(I_j, {\rm\dot H}^1(\mathbb{R}^6))}
+ \Vert \verb"S"(t-t_j) \b{\rm w}(t_j) \Vert_{\mathcal{S}(I_j, {\rm\dot H}^1(\mathbb{R}^6))}
+ C \varepsilon.
\end{equation*}
As a result,
\begin{equation}
\Vert \b{\rm w} \Vert_{\mathcal{S}(I_j, {\rm\dot H}^1(\mathbb{R}^6))}
\le 2\Vert \verb"S"(t-t_j) \b{\rm w}(t_j) \Vert_{\mathcal{S}(I_j, {\rm\dot H}^1(\mathbb{R}^6))}
+ 2C \varepsilon.
\label{p2}
\end{equation}

Now take $ t = t_{j+1} $ in the above Duhamel formula and apply $ \verb"S"(t - t_{j+1}) $ to both sides to obtain
\begin{equation*}
\verb"S"(t - t_{j+1}) \b{\rm w}(t_{j+1}) = \verb"S"(t - t_j) \b{\rm w}(t_j) + i \int_{I_j} \verb"S"(t - s) \b{\rm N}(\b{\rm w}(s)) {\rm d}s.
\end{equation*}
Similarly, we can also get
\begin{equation*}
\begin{aligned}
& ~ \Vert \verb"S"(t - t_{j+1}) \b{\rm w}(t_{j+1}) \Vert_{\mathcal{S}({\rm\dot H}^1(\mathbb{R}^6))} \\
\le & ~ \Vert \verb"S"(t-t_j) \b{\rm w}(t_j) \Vert_{\mathcal{S}({\rm\dot H}^1(\mathbb{R}^6))}
+ 3C \delta \Vert \b{\rm w} \Vert_{\mathcal{S}(I_j, {\rm\dot H}^1(\mathbb{R}^6))}
+ C \Vert \b{\rm w} \Vert_{\mathcal{S}(I_j, {\rm\dot H}^1(\mathbb{R}^6))}^2 + C \varepsilon,
\end{aligned}
\end{equation*}
and
\begin{equation*}
\Vert \verb"S"(t - t_{j+1}) \b{\rm w}(t_j) \Vert_{\mathcal{S}({\rm\dot H}^1(\mathbb{R}^6))}
\le 2 \Vert \verb"S"(t-t_j) \b{\rm w}(t_j) \Vert_{\mathcal{S}({\rm\dot H}^1(\mathbb{R}^6))} + 2C \varepsilon.
\end{equation*}
Iterating the above procedure with the start $ j=0 $, we obtain
\begin{equation*}
\begin{aligned}
\Vert \verb"S"(t - t_j) \b{\rm w}(t_{j+1}) \Vert_{\mathcal{S}({\rm\dot H}^1(\mathbb{R}^6))}
\le & ~ 2^j \Vert \verb"S"(t-t_0) \b{\rm w}(t_0) \Vert_{\mathcal{S}({\rm\dot H}^1(\mathbb{R}^6))} + (2^j - 1) 2C \varepsilon \\
\le & ~  2^{j+2}C \varepsilon.
\end{aligned}
\end{equation*}
To accommodate $ \Vert \verb"S"(t-t_j) \b{\rm w}(t_j) \Vert_{\mathcal{S}(I_j, {\rm\dot H}^1(\mathbb{R}^6))} + C \varepsilon \le \frac{1}{8C} - \frac{3}{2} \delta $ for all intervals $ I_j $, $ 0 \le j \le N-1 $, we require that
\begin{equation*}
2^{N+2}C \varepsilon \le \frac{1}{8C} - \frac{3}{2} \delta,
\end{equation*}
which can be ensured by choosing $0<\varepsilon\ll 1$.
\end{proof}

\section{Variational Characterization}\label{var}
In this section, we are in the position to give the variational characterization for the sharp Gargliardo-Nirenberg inequality.  Apart from stating the existence of the ground state, we will also show some relationships between the energy of solution and the energy of ground state, which facilitate us to obtain the finite time blow-up result in Theorem \ref{classification}.

\subsection{Ground state}\label{ground}
 We first recall some basic properties of ground state. As a quick application,  we obtain the sharp Gargliardo-Nirenberg inequality.
Then we will show the classification in Theorem \ref{classification} is complete.

\begin{proposition}
[Ground state, \cite{Hayashi2013}]\label{ground state}
For $ D_6 = \left\{ \b{\rm G} \in {\rm\dot H}^1 \cap {\rm L}^3 ~ \big\vert ~ R(\b{\rm G}) > 0 \right\} $,
the minimal $ J_{\min} $ of the nonnegative functional
\[
    J(\b{\rm G}) : = \left[ H(\b{\rm G}) \right]^3 \left[ R(\b{\rm G}) \right]^{-2},  \qquad  \b{\rm G} \in D_6
\]
are attained at $ {\b{\rm G}} = (\Phi, \Psi)^T \in \C^2 $,
whose expression has to have the form of $ (\Phi, \Psi) = (e^{i\theta} m \phi(nx), e^{2i\theta} m \varphi(nx)) $,
where $ m > 0 $, $ n > 0 $, $ \theta \in \mathbb R $,
and $ \b{\rm W} := (\phi, \varphi)^T \in D_6 $ is a pair of non-negative real-valued functions satisfying \eqref{gs}.
The function $ {\b{\rm W}} $ is called a {\bf ground state} with $ J({\b{\rm W}}) = J_{\min} $.
The set of all ground states is denoted as $ \mathcal{G} $.

In addition, ground states for \eqref{gs} are unique up to translations and dilations.
The unique ground state is given by $ \b{\rm W} = (\phi_0, \frac{\phi_0}{\sqrt{\kappa}})^T $ where $ \phi_0 $ is a ground state of \eqref{6} in $ {\dot H}^1(\R^6) \cap L^3(\R^6)$.
\end{proposition}

\begin{remark}\label{HR}
For the energy of ground state $ \b{\rm W} = (\phi, \varphi)^T \in D_6 $, we have the following scaling identity:
\begin{equation*}
E \left( \lambda^{\alpha} \b{\rm W}(\lambda^\beta ~ \cdot ~) \right) = \lambda^{2\alpha - 4\beta}H(\b{\rm W}) - \lambda^{3\alpha - 6\beta}R(\b{\rm W}), \qquad \forall ~ \lambda \in (0, \infty).
\end{equation*}
Using variational derivatives and letting  $\lambda = 1$ in both sides of above identity, we can obtain
\[
\begin{aligned}
0 = & ~ \Re \int_{\mathbb{R}^6} \left( -2\Delta \phi -2\phi \varphi \right) \cdot \overline{ \left( \frac{\rm d}{{\rm d}\lambda} \Big|_{\lambda=1} \lambda^{\alpha} \phi(\lambda^\beta x) \right) } \\
& \quad \quad + \left( -\kappa\Delta \varphi -\phi^2 \right) \cdot \overline{ \left( \frac{\rm d}{{\rm d}\lambda} \Big|_{\lambda=1} \lambda^{\alpha} \varphi(\lambda^\beta x) \right) } {\rm d}x   \\
= & ~ (2\alpha - 4\beta) H(\b{\rm W}) - (3\alpha - 6\beta) R(\b{\rm W}) \\
= & ~ [2H(\b{\rm W}) - 3R(\b{\rm W})] \alpha - [4H(\b{\rm W}) - 6R(\b{\rm W})] \beta, \quad \forall ~ \alpha \in \mathbb{R}, ~ \beta \in \R.
\end{aligned}
\]
This yields
\begin{equation}
H(\b{\rm W}) : R(\b{\rm W}) = 3 : 2.
\label{Q}
\end{equation}
This argument of scaling analysis produces the same effect as Pohozaev's identities.
\end{remark}

\begin{corollary}[Gagliardo-Nirenberg inequality]\label{GN}
For any $ \b{\rm g} \in {\rm\dot H}^1(\mathbb{R}^6) $,
\begin{equation}
R(\b{\rm g}) \le C_{GN} [H(\b{\rm g})]^{\frac32},
\label{GNi}
\end{equation}
where
$ C_{GN} = J_{\min}^{-\frac12} > 0 $ is a constant.
Besides, $ R(\b{\rm W}) = C_{GN} [H(\b{\rm W})]^{\frac32} $ for any $ \b{\rm W} \in \mathcal{G} $.
\end{corollary}

\begin{proof}
For $ \b{\rm g} \in {\rm\dot H}^1(\R^6) \backslash D_6 = \left\{ \b{\rm g} \in {\rm\dot H}^1(\R^6) \big\vert R(\b{\rm g}) \le 0 \right\} $, \eqref{GNi} holds obviously since $ C_{GN} > 0 $ and $ H(\b{\rm g}) \ge 0 $.  For $ \b{\rm g} \in D_6 $, by Proposition \ref{ground state}, we have
\[
[H(\b{\rm g})]^3 [R(\b{\rm g})]^{-2} = J(\b{\rm g}) \ge J_{\min},
\]
which leads to \eqref{GNi} directly.

From the proof of Proposition \ref{ground state}, we can easily see that $ R(\b{\rm W}) = C_{GN} [H(\b{\rm W})]^{\frac32} $ for any $ \b{\rm W} \in \mathcal{G} $.
\end{proof}

\begin{proposition}[Coercivity of energy]\label{coer}
Let $ 0 < \delta < 1 $.
For the initial data $ \b{\rm u}_0 \in {\rm\dot H}^1(\R^6) $ and satisfying $ E(\b{\rm u}_0) \le (1 - \delta) E(\b{\rm W}) $ in \eqref{NLS system}, if $ H(\b{\rm u}_0) < H(\b{\rm W}) $, then there exists $ \delta^{\prime\prime} = \delta^{\prime\prime} (\delta) > 0 $ so that
\begin{equation}
2 H(\b{\rm u}) - 3 R(\b{\rm u}) \ge \delta^{\prime\prime} H(\b{\rm u});
\label{coe}
\end{equation}
if $ H(\b{\rm u}_0) > H(\b{\rm W}) $ then there exists $ \tilde\delta^{\prime\prime} = \tilde\delta^{\prime\prime} (\delta) > 0 $ so that
\begin{equation}
2 H(\b{\rm u}) - 3 R(\b{\rm u}) \le - \tilde\delta^{\prime\prime} H(\b{\rm u}).
\label{coe-1}
\end{equation}
\end{proposition}

\begin{proof}
We start with the energy of $ \b{\rm u} $, by \eqref{GNi},
\begin{equation}
E(\b{\rm u}) = H(\b{\rm u}) - R(\b{\rm u}) \ge H(\b{\rm u}) - C_{GN} [H(\b{\rm u})]^{\frac32}.
\label{f}
\end{equation}
 Define $f(y) = y - C_{GN} y^{\frac32} $ for $ y = y(t) := H(\b{\rm u}(t)) \ge 0 $.
Then $ f^\prime(y) = 1 - \frac32 C_{GN} y^{\frac12} $.
Thus $ f^\prime(y) \ge 0 $ only when $ y \le y_0 := \frac4{9C_{GN}^2} $.
 From Remark \ref{HR} and Corollary \ref{GN}, we may obtain $ C_{GN} = \frac{2}{3[H(\b{\rm W})]^{\frac12}}$ and $y_0 = H(\b{\rm W})$.

Therefore,
\begin{equation}
f_{\max}(y) = f(y_0) = f(H(\b{\rm W})) = H(\b{\rm W}) - C_{GN} [H(\b{\rm W})]^{\frac32} = \frac13 H(\b{\rm W}) = E(\b{\rm W}).
\label{fmax}
\end{equation}
 For any time $ t \in I $, by \eqref{f}, \eqref{fmax}, and the conversation of energy, we get
\begin{equation*}
f(H(\b{\rm u}(t))) \le E(\b{\rm u}(t)) = E(\b{\rm u_0}) < E(\b{\rm W}) = f(H(\b{\rm W})),
\end{equation*}
which means $ H(\b{\rm u}(t)) \ne H(\b{\rm W}) $ for any $ t \in I $.

If $ H(\b{\rm u}_0) < H(\b{\rm W}) $, by the continuity of $ f(y) $ and $ y(t) $,  there exists $ \delta^\prime = \delta^\prime (\delta) > 0 $ so that
\begin{equation}
H(\b{\rm u}(t)) \le (1 - \delta^\prime) H(\b{\rm W}), \quad \forall ~ t \in I.
\label{<}
\end{equation}
Similarly, if $ H(\b{\rm u}_0) > H(\b{\rm W}) $, there exists $ \tilde\delta^\prime = \tilde\delta^\prime (\delta) > 0 $ so that
\begin{equation}
H(\b{\rm u}(t)) \ge (1 + \tilde\delta^\prime) H(\b{\rm W}), \quad \forall ~ t \in I.
\label{>}
\end{equation}

Furthermore, there exist $\delta^{\prime\prime} = \delta^{\prime\prime} (\delta^\prime) > 0 $ and $ \tilde\delta^{\prime\prime} = \tilde\delta^{\prime\prime} (\delta, \tilde\delta^\prime) > 0 $ so that
\[
2 - 3 \frac{R(\b{\rm u})}{H(\b{\rm u})} \ge 2 - 3 C_{GN}[H(\b{\rm u})]^{\frac12} = 2 - 2 \left( \frac{H(\b{\rm u})}{H(\b{\rm W})} \right)^{\frac12} \ge 2 \left[ 1 - \left( 1 - \delta^\prime \right)^{\frac12} \right] := \delta^{\prime\prime};
\]
and
\[
2 - 3 \frac{R(\b{\rm u})}{H(\b{\rm u})} = 3 \frac{E(\b{\rm u}_0)}{H(\b{\rm u})} - 1 \le 3 \frac{(1 - \delta)E(\b{\rm W})}{(1 + \delta^\prime)H(\b{\rm W})} - 1 = \frac{(1 - \delta)}{(1 + \delta^\prime)} - 1 := - \tilde\delta^{\prime\prime},
\]
that is,
\begin{equation}
2 H(\b{\rm u}) - 3 R(\b{\rm u}) \ge \delta^{\prime\prime} H(\b{\rm u});
\label{coe''}
\end{equation}
and
\begin{equation}
2 H(\b{\rm u}) - 3 R(\b{\rm u}) \le - \tilde\delta^{\prime\prime} H(\b{\rm u}).
\label{coe'''}
\end{equation}
So we complete the proof.
\end{proof}

\subsection{Energy trapping}\label{et}
 For the solution to classical focusing \eqref{NLS}, generally speaking, we can not determine the sign of its energy. However, under certain circumstance, we may determine the sign of a general function regardless of whether it is a solution to \eqref{NLS} or not. For us, we may generalize the above fact to our setting.

\begin{lemma}\label{E>0}
For any $ \b{\rm v} \in {\rm\dot H}^1(\mathbb{R}^6) $, if $ H(\b{\rm v}) \le H(\b{\rm W}) $, then
\[
E(\b{\rm v}) \ge 0.
\]
\end{lemma}

\begin{proof}
 We just need to show $ H(\b{\rm v}) - C_{GN} [H(\b{\rm v})]^{\frac32} \ge 0 $, since
\[
E(\b{\rm v}) = H(\b{\rm v}) - R(\b{\rm v}) \ge H(\b{\rm v}) - C_{GN} [H(\b{\rm v})]^{\frac32}.
\]
Considering $ 0 \le H(\b{\rm v}) \le H(\b{\rm W}) $, we define $ f(y) = y - C_{GN} y^{\frac32} $ for $ 0 \le y \le H(\b{\rm W}) $.
Then $ f^\prime(y) = 1 - \frac32 C_{GN} y^{\frac12} $.
Thus $ f^\prime(y) \ge 0 $ only when $ y \le y_0 := \frac4{9C_{GN}^2} $.
 The fact $ C_{GN} = \frac{R(\b{\rm W})}{[H(\b{\rm W})]^{\frac32}} $ and $ H(\b{\rm W}) : R(\b{\rm W}) = 2 : 3 $ implies $$ C_{GN} = \frac{2}{3[H(\b{\rm W})]^{\frac12}} \quad \text{and}\quad  y_0 = H(\b{\rm W}) .$$
Therefore, we obtain $ E(\b{\rm v}) \ge f(H(\b{\rm v})) \ge 0 $ if $ 0 \le H(\b{\rm v}) \le H(\b{\rm W}) $.
\end{proof}

 Next, we will establish the equivalence relation between the energy and the kinetic energy of the solution to \eqref{NLS system}, that is, $ H(\b{\rm u}) \sim E(\b{\rm u}) $ under certain restrictions. For the classical \eqref{NLS}, it is easy to obtain $ E(\b{\rm u}) \lesssim H(\b{\rm u}) $. Indeed,
for the focusing case, we can use the fact that $ R(\b{\rm u}) \ge 0 $, which is missing in our setting;
for the defocusing case, we can use the associated Gagliardo-Nirenberg inequality. However,  for \eqref{NLS system}, we do not have the natural comparison, $ E(\b{\rm u}) \lesssim H(\b{\rm u}) $.
Our proof here is different from the argument in \cite{Kenig2006}.
To obtain $ \vert R(\b{\rm u}) \vert \lesssim_{\kappa} [H(\b{\rm u})]^{\frac32} $,  the Arithmetic-Geometry Mean-value inequality is used.
Combining with the coercivity of energy, we prove $ H(\b{\rm u}) \sim E(\b{\rm u})$.

\begin{lemma}\label{R}
For any $ \b{\rm v} \in {\rm\dot H}^1(\mathbb{R}^6) $,
\begin{equation}
\vert R(\b{\rm v}) \vert \le C(\kappa) [H(\b{\rm v})]^{\frac32},
\label{||}
\end{equation}
where $ C(\kappa) = \sqrt{\frac{8}{27\kappa}} > 0 $.
\end{lemma}

\begin{proof}
 Using H\"older's  inequality, Sobolev embedding, and the Arithmetic-Geometry Mean-value inequality,
\[
\frac1n \sum_{i=1}^n a_i \ge \sqrt[n]{\prod_{i=1}^n a_i}, \qquad \text{for} ~ \forall ~ 1 \le i \le n, a_i \ge 0,
\]
we can compute directly, setting $ \b{\rm v} := (\tilde{u}, \tilde{v})^T$,
\begin{equation*}
\begin{aligned}
\vert R(\b{\rm v}) \vert = & ~ \left\vert \Re \int_{\R^6} \overline{\tilde{v}} \tilde{u}^2 {\rm d}x \right\vert
\le \int_{\R^6} \vert \tilde{v} \vert \vert \tilde{u} \vert^2 {\rm d}x \\
\le & ~ \Vert \tilde{v} \Vert_{L^3(\R^6)} \Vert \tilde{u} \Vert_{L^3(\R^6)}^2
\le \Vert \tilde{v} \Vert_{{\dot H}^1(\R^6)} \Vert \tilde{u} \Vert_{{\dot H}^1(\R^6)}^2 \\
= & ~ \left[ \frac{8}{\kappa} \left( \frac{\kappa}{2} \Vert \tilde{v} \Vert_{{\dot H}^1(\R^6)}^2 \right) \left( \frac{1}{2} \Vert \tilde{u} \Vert_{{\dot H}^1(\R^6)}^2 \right) \left( \frac{1}{2} \Vert \tilde{u} \Vert_{{\dot H}^1(\R^6)}^2 \right) \right]^{\frac12} \\
\le & ~ \left\{ \frac{8}{\kappa} \left[ \frac{\left( \frac{\kappa}{2} \Vert \tilde{v} \Vert_{{\dot H}^1(\R^6)}^2 \right) + \left( \frac{1}{2} \Vert \tilde{u} \Vert_{{\dot H}^1(\R^6)}^2 \right) + \left( \frac{1}{2} \Vert \tilde{u} \Vert_{{\dot H}^1(\R^6)}^2 \right)}{3} \right]^3 \right\}^{\frac12} \\
= & ~ \sqrt{\frac{8}{27\kappa}} [H(\b{\rm v})]^{\frac32}.
\end{aligned}
\end{equation*}
Choosing $ C(\kappa) = \sqrt{\frac{8}{27\kappa}} $, we obtain \eqref{||}.
\end{proof}

Using Lemma \ref{R} and Proposition \ref{coer},  we can establish  the energy trapping result.

\begin{corollary}[Energy Trapping]\label{Energy trapping}
Let $ \b{\rm u} $ be a solution to \eqref{NLS system} with initial data $ \b{\rm u}_0 $ and maximal life-span $ I_{\max} \ni 0 $.
If $ E(\b{\rm u}_0) \le (1 - \delta) E(\b{\rm W}) $ and $ H(\b{\rm u}_0) \le (1 - \delta^{\prime}) H(\b{\rm W}) $, then
\begin{equation}
H(\b{\rm u}(t)) \sim E(\b{\rm u}(t)), \qquad \forall ~ t \in I_{\max}.
\label{EH}
\end{equation}
\end{corollary}

\begin{proof}
By \eqref{||} and \eqref{<}, we may get the upper bound of energy $ E(\b{\rm u})$,
\begin{equation*}
\begin{aligned}
E(\b{\rm u}(t)) \le & ~ H(\b{\rm u}(t)) + \vert R(\b{\rm u}(t)) \vert \\
\le & ~ H(\b{\rm u}(t)) + C(\kappa) \left[ H(\b{\rm u}(t)) \right]^{\frac{3}{2}} \\
\le & ~ \left( 1 + C(\kappa) \left[ (1 - \delta^\prime) H(\b{\rm W}) \right]^{\frac12} \right) H(\b{\rm u}(t)).
\end{aligned}
\end{equation*}
As for the lower bound of $ E(\b{\rm u}) $, the proof of Proposition \ref{coer} tells us
\begin{equation*}
\begin{aligned}
E(\b{\rm u}(t))
= & ~ \frac{1}{3} H(\b{\rm u}(t)) + \frac{1}{3} \left[ 2H(\b{\rm u}(t)) - 3R(\b{\rm u}(t)) \right] \\
\ge & ~ \frac{1}{3} H(\b{\rm u}(t)) + \frac{1}{3}\delta^{\prime\prime} H(\b{\rm u}(t)) \\
= & ~ \frac{1}{3} (1 + \delta^{\prime\prime}) H(\b{\rm u}(t)).
\end{aligned}
\end{equation*}
Combining the above two inequalities, we deduce the conclusion \eqref{EH}.
\end{proof}


\subsection{Virial identity and blowing-up}\label{idea}
Despite being proved  in \cite{Dihn2020}, the part of blowing-up result of Theorem \ref{classification} will also be shown for completeness.

 We aim to figure out how the mass is distributed in the spatial space. In particular, we try to make clear whether it concentrates at the origin or not,
and how the distribution changes over time.
So we consider a class of initial data satisfying $ x \b{\rm u_0} \in {\rm L}^2(\R^6) $ firstly.
The law of conservation of mass suggests it useless to think about the change of $ \Vert \b{\rm u}(t) \Vert_{{\rm L}^2} $ over time directly,
so we redistribute the mass of $ \b{\rm u} $ to large radii
and consider the behavior of $ \Vert x \b{\rm u}(t) \Vert_{{\rm L}^2} $ when $ t \in I_{max} $.
For $ d = 6 $ in the \eqref{NLS system} we define a function
\begin{equation}
I(t) := \int_{\mathbb{R}^6} \vert x \vert^2 \left( 2 \kappa \vert u \vert^2 + \vert v \vert^2 \right) {\rm d}x.
\label{I}
\end{equation}

\begin{remark}
 The coefficients in  \eqref{I} are chosen carefully to match with Propostion \ref{coer}, one may refer to Remark 4.1 in \cite{Meng2021} for details.
More specifically, just by adjusting the rate of coefficients between $ \vert u \vert^2 $ and $ \vert v \vert^2 $ to be $ 2 \kappa : 1 $, we can obtain a factor $ 2 H(\b{\rm u}) - 3 R(\b{\rm u}) $  in \eqref{I''}.
\end{remark}

\begin{lemma}[Virial identity]\label{virial}
If $ \b{\rm u} $ is a solution to \eqref{NLS system}, for any real valued function $ a \in C^\infty(\mathbb{R}^6) $, $ t \in [0, T_+(\b{\rm u})) $, then \\
{\rm (1)}
\[
\frac{\rm d}{{\rm d}t} \int_{\mathbb{R}^6} \left( 2 \kappa \vert u \vert^2 + \vert v \vert^2 \right) a(x) {\rm d}x = 2 \kappa \Im \int_{\mathbb{R}^6} \left( 2 \overline{u} \nabla u + \overline{v} \nabla v \right) \cdot \nabla a(x) {\rm d}x.
\]
{\rm (2)}
\[
\begin{aligned}
& ~ \frac{{\rm d}^2}{{\rm d}t^2} \int_{\mathbb{R}^6} \left( 2 \kappa \vert u \vert^2 + \vert v \vert^2 \right) a(x) {\rm d}x \\
= & ~ 8 \kappa \Re \int_{\mathbb{R}^6} \left( \overline{u_j} u_{k} + \frac{\kappa}{2} \overline{v_j} v_{k} \right) a_{jk}(x) {\rm d}x - 2 \kappa \int_{\mathbb{R}^6} \left( \vert u \vert^2 + \frac{\kappa}{2} \vert v \vert^2 \right) \Delta \Delta a(x) {\rm d}x \\
& ~ - 2 \kappa \Re \int_{\mathbb{R}^6} \overline{v} u^2 \Delta a(x) {\rm d}x.
\end{aligned}
\]
\end{lemma}
 Set $a(x)=x^2$, a simple calculation shows that\[
a(x) = \vert x \vert^2, \qquad  a_j(x):=\partial_{x_j}a = 2 x_j, \qquad \nabla a(x) = 2 x,
\]
and
\[
a_{jk}(x):=\partial_{x_i}\partial_{x_j}a = 2 \delta_{jk}, \qquad \Delta a(x) = 2d, \qquad \Delta \Delta a(x) = 0.
\]
As a result, by Lemma \ref{virial}, we have
\begin{equation*}
I^{\prime}(t) = 4 \kappa \Im \int_{\mathbb{R}^6} \left( 2 \overline{u} \nabla u + \overline{v} \nabla v \right) \cdot x {\rm d}x
\end{equation*}
and
\begin{equation}
I^{\prime\prime}(t) = 16\kappa H(\b{\rm u}) - 24\kappa R(\b{\rm u}) = 8\kappa [ 2 H(\b{\rm u}) - 3 R(\b{\rm u}) ].
\label{I''}
\end{equation}

By {\color{red}the } Gagliardo-Nirenberg inequality \eqref{GNi}, we know
\begin{equation*}
\begin{aligned}
\frac1{8\kappa} I^{\prime\prime}(t) = & ~ 2 H(\b{\rm u}) - 3 R(\b{\rm u}) \\
\ge & ~ 2 H(\b{\rm u}) - 2 \left[ H(\b{\rm W}) \right]^{-\frac{1}{2}} \left[ H(\b{\rm u}) \right]^{\frac{3}{2}} \\
= & ~ 2 \left[ H(\b{\rm u}) \right]^{\frac{3}{2}} \left( \left[ H(\b{\rm u}) \right]^{-\frac{1}{2}} - \left[ H(\b{\rm W}) \right]^{-\frac{1}{2}} \right).
\end{aligned}
\end{equation*}
So under the hypotheses of $ E(\b{\rm u_0}) < E(\b{\rm W}) $ and $ H(\b{\rm u_0}) < H(\b{\rm W}) $, we have $ I^{\prime\prime}(0) > 0 $,
which implies $ I(t) $ is a convex function.
We can expect $ \b{\rm u} $ is global and scatters.

Under the hypotheses of $ E(\b{\rm u_0}) < E(\b{\rm W}) $ and $ H(\b{\rm u_0}) > H(\b{\rm W}) $, Proposition \ref{coer} tells us that the right side of \eqref{I''} is strictly negative,
which implies $ I(t) $ is a concave function.
 Combing with $ I(0) >0$, we know the variance $ I(t) $ will tend to $0$ in finite time,
which means that all the mass of $ \b{\rm u} $ concentrates at the origin and  $ \b{\rm u} $ has to blow up in finite time.

The coercivity of energy has showed some relations between the potential energy and the kinetic energy of solution in any time $ t \in I_{\max} $.
But we need some uniform bounds sometimes.

\begin{theorem}\label{vs}
For $ x\b{\rm u}_0 \in {\rm L}^2(\mathbb{R}^6) $, we set $ I_{\max} $ be the maximal time interval of existence of $ \b{\rm u}(t) $ solving \eqref{NLS system}.
If $ E(\b{\rm u}_0) < E(\b{\rm W}) $ and $ H(\b{\rm u}_0) > H(\b{\rm W}) $, then $ I_{\max} $ is finite and $ H(\b{\rm u}(t)) > H(\b{\rm W}), ~ \forall ~ t \in I_{\max} $.
\end{theorem}

\begin{proof}
If $ H(\b{\rm u}_0) > H(\b{\rm W}) $, by the properties of $ f(y) $ as in the proof of Proposition \ref{coer} and \eqref{>}, we can choose $ \tilde\delta^\prime > 0 $ such that
\[
H(\b{\rm u}(t)) \ge (1 + \tilde\delta^\prime) H(\b{\rm W}), ~ \forall ~ t \in I_{\max}.
\]
And then,
\begin{equation*}
\begin{aligned}
2 H(\b{\rm u}(t)) - 3 R(\b{\rm u}(t)) = & ~ 3 E(\b{\rm u}(t)) - H(\b{\rm u}(t)) \\
= & ~ 3 E(\b{\rm u}_0) - H(\b{\rm u}(t)) \\
\le & ~ 3 E(\b{\rm W}) - (1 + \tilde\delta^\prime) H(\b{\rm W}) \\
= & ~ H(\b{\rm W}) - (1 + \tilde\delta^\prime) H(\b{\rm W})
= -\tilde\delta^\prime H(\b{\rm W}).
\end{aligned}
\end{equation*}
The Virial identity \eqref{I''} shows that
\[
I^{\prime\prime}(t) = 8\kappa [ 2 H(\b{\rm u}) - 3 R(\b{\rm u}) ] \le - 8\kappa \tilde\delta^\prime H(\b{\rm W}), ~ \forall ~ t \in I_{\max}.
\]
So, the time $ t \in I_{\max} $ must be finite.
\end{proof}

\begin{theorem}\label{bu}
For $ \b{\rm u}_0 \in {\rm H}^1(\R^6) $ radial instead of the assumption $ x\b{\rm u}_0 \in {\rm L}^2(\R^6) $ in Theorem \ref{vs}.
If the other conditions keep the same as Theorem \ref{vs}, then the same result holds as well.
\end{theorem}

\begin{proof}
We change the weighted function $ \vert x \vert^2 $ into a specific smooth cut-off function in \eqref{I}, to be more specific,
\[
\tilde{I}(t) := \int_{\mathbb{R}^6} \left( 2 \kappa \vert u \vert^2 + \vert v \vert^2 \right) a(x) {\rm d}x, \quad a(x) := R^2 \Gamma \left( \frac{\vert x \vert^2}{R^2} \right), ~ \forall ~ R > 0,
\]
where $ \Gamma(r) $ is a smooth concave function defined on $ [0, \infty) $ satisfying
\[
\Gamma(r) = \left\{
\begin{aligned}
r, \qquad r \le 1, \\
2, \qquad r \ge 3,
\end{aligned}
\right. \qquad \text{and} \qquad
\left\{
\begin{aligned}
\Gamma^{\prime\prime}(r) \searrow, \qquad r \le 2, \\
\Gamma^{\prime\prime}(r) \nearrow, \qquad r \ge 2.
\end{aligned}
\right.
\]
By Lemma \ref{virial}, we have
\[
\begin{aligned}
I^{\prime\prime}(t) = & ~ 8\kappa \int_{\R^6} [2H(\b{\rm u}) - 3R(\b{\rm u})] {\rm d}x + \frac{1}{R^2} O \left( \int_{\vert x\vert \sim R} \vert \b{\rm u} \vert^2 {\rm d}x \right) \\
& ~ \quad + 8\kappa \int_{\R^6} \left( \Gamma^\prime \left( \frac{\vert x \vert^2}{R^2} \right) - 1 + \frac{2\vert x \vert^2}{R^2} \Gamma^{\prime\prime} \left( \frac{\vert x \vert^2}{R^2} \right) \right) [2H(\b{\rm u}) - 3R(\b{\rm u})] {\rm d}x \\
& ~ \quad - \frac{40}{3} \Re \int_{\R^6} \frac{2\vert x \vert^2}{R^2} \Gamma^{\prime\prime} \left( \frac{\vert x \vert^2}{R^2} \right) \overline{v} u^2 {\rm d}x.
\end{aligned}
\]
Because $ \Gamma^{\prime\prime}(r) > 0 $ and $ \b{\rm u} \in {\rm L}^2(\R^6) $, we can choose $ R = R(M(\b{\rm u}))) $ large enough to obtain
\[
I^{\prime\prime}(t) \le - 4\kappa \tilde\delta^\prime H(\b{\rm W}) - 8\kappa \int_{\R^6} \omega(x) [2H(\b{\rm u}) - 3R(\b{\rm u})] {\rm d}x,
\]
where
\[
\omega(x) = 1 - \Gamma^\prime \left( \frac{\vert x \vert^2}{R^2} \right) - \frac{2\vert x \vert^2}{R^2} \Gamma^{\prime\prime} \left( \frac{\vert x \vert^2}{R^2} \right).
\]
Observe that $ 0 \le \Gamma \le 1 $ is radial, $ {\rm supp}\;(\Gamma) \subset \left\{ x \big\vert \vert x \vert > R \right\} $, and
\[
\Gamma(x) \lesssim \Gamma(y), \;\; \text{uniformly for} ~ \vert x \vert \le \vert y \vert,
\]
we have
\[
\left\Vert \vert x \vert^{\frac52} \Gamma^{\frac14} f \right\Vert_{L_x^\infty(\R^6)}^2 \lesssim \left\Vert f \right\Vert_{L_x^2(\R^6)} \left\Vert \Gamma^{\frac12} \nabla f \right\Vert_{L_x^2(\R^6)}.
\]
Furthermore, by the conservation of mass,
\[
\begin{aligned}
\Re \int_{\R^6} \Gamma(x) \overline{v}u^2 {\rm d}x \lesssim & ~ \left\Vert \Gamma^{\frac14} \b{\rm u}(t) \right\Vert_{{\rm L}_x^\infty(\R^6)} \int_{\R^6} \vert \b{\rm u}(t, x) \vert^2 {\rm d}x \\
\lesssim & ~ R^{-\frac52} \left\Vert \vert x \vert^{\frac52} \Gamma^{\frac14} \b{\rm u}(t) \right\Vert_{{\rm L}_x^\infty(\R^6)} M(\b{\rm u}) \\
\lesssim & ~ R^{-\frac52} \left\Vert \Gamma^{\frac12} \nabla \b{\rm u}(t) \right\Vert_{{\rm L}_x^2(\R^6)}^{\frac12} [M(\b{\rm u})]^{\frac54}.
\end{aligned}
\]
This implies that $ \tilde{I}^{\prime\prime}(t) < 0 $ by choosing $ R > 0 $ large enough.
\end{proof}

Theorem \ref{vs} and Theorem \ref{bu} have rigorously proved some original ideas in Subsection \ref{idea} and showed the blowing-up result in Theorem \ref{classification}.

\section{Proof of Theorem \ref{minH}}\label{LPD}
In this section, we use the inverse Strichartz inequality to derive a linear profile decomposition for the Schr\"odinger propagator $ \verb"S"(t) $.
Combining with the nonlinear profile, we deduce the existence of critical solution.

\subsection{Linear profile decomposition}\label{subsec:lpd}
In this subsection, we show the profile decomposition with scaling parameter of a radial uniformly bounded sequence in $ {\rm\dot H}^1(\mathbb{R}^6) $.
We start by combining the Strichartz inequality for the Schr\"odinger propagator $ \verb"S"(t) $, Theorem \ref{strichartz}, and Sobolev embedding to obtain
\begin{equation}
\left\Vert \verb"S"(t) \b{\rm g} \right\Vert_{{\rm L}_{t, x}^4(\mathbb{R} \times \mathbb{R}^6)} \lesssim \left\Vert \verb"S"(t) \nabla \b{\rm g} \right\Vert_{{\rm L}_{t, x}^4(\mathbb{R} \times \mathbb{R}^6)} \lesssim \Vert \b{\rm g} \Vert_{{\rm\dot H}_x^1(\mathbb{R}^6)}.
\label{Str}
\end{equation}

Our next result is a refinement of \eqref{Str}, which says that if the linear evolution of $ \b{\rm g} $ is large in $ {\rm L}_{t, x}^4(\mathbb{R} \times \mathbb{R}^6) $, then the linear evolution of a single Littlewood-Paley piece of $ \b{\rm g} $ is, at least partially, responsible.

\begin{lemma}[Refined Strichartz estimate]\label{RSE}
For any function $ \b{\rm g} \in {\rm\dot H}^1(\mathbb{R}^6) $, we have
\begin{equation*}
\left\Vert \verb"S"(t) \b{\rm g} \right\Vert_{{\rm L}_{t, x}^4(\mathbb{R} \times \mathbb{R}^6)}
\lesssim \Vert \b{\rm g} \Vert_{{\rm\dot H}^1(\mathbb{R}^6)}^{\frac12} \sup_{N \in 2^{\mathbb{Z}}} \left\Vert \verb"S"(t) P_{N} \b{\rm g} \right\Vert_{{\rm L}_{t, x}^4(\mathbb{R} \times \mathbb{R}^6)}^{\frac12}.
\end{equation*}
\end{lemma}

\begin{proof}
From the square function estimate and the Bernstein and Strichartz inequality, we can obtain the result above.
\end{proof}

The refined Strichartz estimate and Lemma \ref{RSE} demonstrate that a linear solution to \eqref{NLS system} with a nontrivial space-time norm will be Fourier concentrated in at least one circular area. The following proposition will further manifest that a linear solution contains a concentrated ``bubble" around a certain space-time.

\begin{proposition}[Inverse Strichartz inequality]\label{ISI}
If the sequence $ \{ \b{\rm g}_n \}_{n=1}^\infty \subset {\rm\dot H}^1(\mathbb{R}^6) $ satisfies
\begin{equation*}
\lim_{n \to \infty} \Vert \b{\rm g}_n \Vert_{{\rm\dot H}^1(\mathbb{R}^6)} = A \le \infty, \quad \text{and} \quad \lim_{n \to \infty} \left\Vert \verb"S"(t) \b{\rm g}_n \right\Vert_{{\rm L}_{t, x}^4(\mathbb{R} \times \mathbb{R}^6)} = \eps > 0.
\end{equation*}
Then there exists a subsequence of $ \{ n \} ($still denoted by $ \{ n \})$, $ \b{\rm \phi} \in {\rm\dot H}^1(\R^6) $, $ \{ \lambda_n \}_{n=1}^\infty \subset (0, \infty) $ and $ \{ (t_n, x_n) \}_{n=1}^\infty \subset \mathbb{R} \times \mathbb{R}^6 $ such that
\begin{equation}
\lambda_n^2 \left[ \verb"S"(t_n) \b{\rm g}_n \right] (\lambda_n x + x_n) \rightharpoonup \b{\rm \phi} (x) \quad \text{weakly in $ {\rm\dot H}_x^1(\mathbb{R}^6) $},
\end{equation}
\begin{equation}
\liminf_{n \to \infty} \left\{ \Vert \b{\rm g}_n \Vert_{{\rm\dot H}_x^1(\mathbb{R}^6)}^2 - \Vert \b{\rm g}_n - \b{\rm \phi}_n \Vert_{{\rm\dot H}_x^1(\mathbb{R}^6)}^2 \right\}
= \Vert \b{\rm \phi} \Vert_{{\rm\dot H_x^1(\mathbb{R}^6)}}^2
\gtrsim A^{-10} \eps^{12},
\end{equation}
\begin{equation}
\liminf_{n \to \infty} \left\{ \left\Vert \verb"S"(t) \b{\rm g}_n \right\Vert_{{\rm L}_{t, x}^4(\mathbb{R} \times \mathbb{R}^6)}^4  - \left\Vert \verb"S"(t) ( \b{\rm g}_n - \b{\rm \phi}_n ) \right\Vert_{{\rm L}_{t, x}^4(\mathbb{R} \times \mathbb{R}^6)}^4 \right\}
\gtrsim A^{-20} \eps^{24},
\end{equation}
where
\begin{equation}
\b{\rm \phi}_n (x) := \frac1{\lambda_n^2} \left[ \verb"S" \left( -\frac{t_n}{\lambda_n^2} \right) \b{\rm \phi} \right] \left( \frac{x - x_n}{\lambda_n} \right).
\end{equation}
\end{proposition}

\begin{remark}
Under the hypotheses of Proposition \ref{ISI} and passing to a further subsequence if necessary, the decoupling of potential energy holds (using Rellich-Kondrashov and refined Fatou), that is,
\begin{equation*}
\liminf_{n \to \infty} \left\{ \Vert \b{\rm g}_n \Vert_{{\rm L}_x^3(\mathbb{R}^6)}^3 - \Vert \b{\rm g}_n - \b{\rm \phi}_n \Vert_{{\rm L}_x^3(\mathbb{R}^6)}^3 - \left\Vert \verb"S" \left( - \frac{t_n}{\lambda_n^2} \right) \b{\rm \phi} \right\Vert_{{\rm L}_x^3(\mathbb{R}^6)}^3 \right\} = 0.
\end{equation*}
\end{remark}

Using the above proposition, we can obtain the linear profile decomposition as follows:

\begin{theorem}
[Linear profile decomposition]\label{lpd}
Let $ \b{\rm \phi}_n = ( \phi_n, \psi_n ) $ be an uniformly bounded sequence in $ {\rm\dot H}^1(\mathbb{R}^6) $ with $ \Vert \b{\rm \phi}_n \Vert \le A $ for any $ 1 \le n < +\infty $.
Then there exists $ M^* \in \{ 1, 2, \cdot\cdot\cdot \} \cup \{ \infty \} $ such that for each finite $ 1 \le M \le M^* $,
\begin{enumerate}
  \item[$(1)$]for each $ 1 \le j \le M $, there exist a sequence of space shifts $ \{ x^j_n \}_{n=1}^\infty \subset \mathbb{R}^6 $, a sequence of time shifts $ \{ t^j_n \}_{n=1}^\infty \subset \mathbb{R} $,
a sequence of scaling shifts $ \{ \lambda^j_n \}_{n=1}^\infty \subset (0, +\infty) $,
and a profile $ \b{\rm \phi}^j := ( \phi^j, \psi^j ) $ (fixed in $ n $) in $ {\rm\dot H}^1(\mathbb{R}^6) $, with
\begin{equation}
\frac{\lambda^j_n}{\lambda^l_n} + \frac{\lambda^l_n}{\lambda^j_n} + \frac{\vert x^j_n - x^l_n \vert^2}{\lambda^j_n \lambda^l_n} + \frac{\vert t^j_n (\lambda^j_n)^2 - t^l_n (\lambda^l_n)^2 \vert}{\lambda^j_n \lambda^l_n} \to + \infty, \quad \forall ~ j \ne l \in \{ 1, 2, \cdot\cdot\cdot, M \}.
\label{asy}
\end{equation}
We may additionally assume that for each $ j $ either $ t^j_n \equiv 0 $ or $ t^j_n \to \pm \infty $.
Especially, for fixed $ 1 \le j_0 \le M $, there exists $ \alpha(j_0) > 0 $ such that
\begin{equation}
H \left( \b{\rm \phi}^{j_0} \right) \ge \alpha(j_0).
\label{H>0}
\end{equation}

    \item[$(2)$]There exists a sequence (in $ n $) of remainders $ \b{\rm \Phi}^M_n := ( \Phi^M_n, \Psi^M_n ) $ in $ {\rm\dot H}^1(\mathbb{R}^6) $,
such that
\begin{equation}
\b{\rm \phi}_n = \sum_{j=1}^M \mathcal{T}_{\lambda^j_n} \left[ \verb"S" \left( \frac{-t^j_n}{(\lambda^j_n)^2} \right) \b{\rm \phi}^j \right] + \b{\rm \Phi}^M_n,
\label{pro}
\end{equation}
where $ \mathcal{T}_\lambda \b{\rm u} = (T_\lambda u, T_\lambda v)^t $ for $ \b{\rm u} = (u, v)^t $, with
\begin{equation}
\lim_{M \to M^*} \limsup_{n \to \infty} \left\Vert \verb"S"(t) \b{\rm \Phi}^M_n \right\Vert_{{\rm L}_{t, x}^4(\mathbb{R} \times \mathbb{R}^6)}
= 0,
\label{remainder1}
\end{equation}
and
\begin{equation}
\verb"S"(-t^M_n) \left[ \mathcal{T}_{\lambda^M_n}^{-1} \b{\rm \Phi}^M_n \right] \rightharpoonup 0 \qquad \text{weakly in $ {\rm\dot H}^1(\mathbb{R}^6) $},
\label{remainder2}
\end{equation}
where $ T_{\lambda}f(x) := \frac1{\lambda^2} f(\frac{x - x^j_n}\lambda) $ and so $ T_{\lambda}^{-1}f(x) := \lambda^2 f(\lambda x + x^j_n) $.

\end{enumerate}
Besides, we have the asymptotic pythagorean expansion:
\begin{equation}
\left\Vert \phi_n \right\Vert_{{\dot H}^1(\mathbb{R}^6)}^2
= \sum_{j=1}^M \left\Vert \phi^j \right\Vert_{{\dot H}^1(\mathbb{R}^6)}^2 + \left\Vert \Phi^M_n \right\Vert_{{\dot H}^1(\mathbb{R}^6)}^2 + o_n (1),
\label{ape1}
\end{equation}
\begin{equation}
\left\Vert \psi_n \right\Vert_{{\dot H}^1(\mathbb{R}^6)}^2
= \sum_{j=1}^M \left\Vert \psi^j \right\Vert_{{\dot H}^1(\mathbb{R}^6)}^2 + \left\Vert \Psi^M_n \right\Vert_{{\dot H}^1(\mathbb{R}^6)}^2 + o_n (1),
\label{ape2}
\end{equation}
\begin{equation}
\left\Vert \phi_n \right\Vert_{L^3(\mathbb{R}^6)}^3
= \sum_{j=1}^M \left\Vert e^{it^j_n\Delta}\phi^j \right\Vert_{L^3(\mathbb{R}^6)}^3 + \left\Vert \Phi^M_n \right\Vert_{L^3(\mathbb{R}^6)}^3 + o_n (1),
\label{ape3}
\end{equation}
\begin{equation}
\left\Vert \psi_n \right\Vert_{L^3(\mathbb{R}^6)}^3
= \sum_{j=1}^M \left\Vert e^{\kappa it^j_n\Delta}\psi^j \right\Vert_{L^3(\mathbb{R}^6)}^3 + \left\Vert \Psi^M_n \right\Vert_{L^3(\mathbb{R}^6)}^3 + o_n (1).
\label{ape4}
\end{equation}
\end{theorem}

At the same, we have the orthogonality of the energy of linear profiles.

\begin{corollary}
[Energy pythagorean expansion]\label{epe}
Under the assumptions of in Theorem \ref{lpd}, we have
\begin{equation}
H(\b{\rm \phi}_n) = \sum_{j=1}^M H \left( \b{\rm \phi}^j \right)
+ H \left( \b{\rm \Phi}^M_n \right) + o_n (1),
\label{HPe}
\end{equation}
and
\begin{equation}
E(\b{\rm \phi}_n) = \sum_{j=1}^M E \left( \verb"S"\left(\frac{-t^j_n}{(\lambda^j_n)^2}\right) \b{\rm \phi}^j \right)
+ E \left( \b{\rm \Phi}^M_n \right) + o_n (1).
\label{EPe}
\end{equation}
\end{corollary}

\subsection{Nonlinear profile}\label{np}
First, we introduce the definition of nonlinear profile, which is associated with an initial data and a sequence of time, and then prove that it is well-defined in this subsection.

\begin{definition}\label{nlp}
Let $ \b{\rm v}_0 \in {\dot{\rm H}}^1(\mathbb{R}^6), \b{\rm v}(t, x) = \verb"S"(t) \b{\rm v}_0 $ and let $ \{ t_n \}_{n=1}^\infty $ be a sequence with the limit $ t_n \to t_{\infty} \in \mathbb{R} \cup \{ \pm \infty \} $.
We say that $ \b{\rm u} $ is a \emph{nonlinear profile} associated with $ (\b{\rm v}_0, \{ t_n \}_{n=1}^\infty) $ if there exists an interval $ I \ni t_{\infty} $ ( if $ \vert t_{\infty} \vert = \infty $, $ I = [a, +\infty) $ or $ I = (-\infty, a] $ ) such that $ \b{\rm u} $ is a solution to \eqref{NLS system} in $ I $ and
\begin{equation*}
\lim_{n \to \infty} \Vert \b{\rm u}(t_n, \cdot) - \b{\rm v}(t_n, \cdot) \Vert_{{\dot{\rm H}}^1(\mathbb{R}^6)} = 0.
\end{equation*}
\end{definition}

By the small data theory and Theorem \ref{sd}, we can prove the existence and uniqueness of the nonlinear profile, which ensures the notation of the nonlinear profile is well-defined  in Definition \ref{nlp}.

\begin{proposition}
There exists at least one nonlinear profile associated with $ (\b{\rm v}_0, \{ t_n \}_{n=1}^\infty) $.
\end{proposition}

\begin{proof}
 There are two cases to consider: either $ \vert t_{\infty} \vert < \infty $ or $ \vert t_{\infty} \vert = \infty $.

{\bf Case 1.} If $ \vert t_{\infty} \vert < \infty $, we take $ \b{\rm u}_0 = \b{\rm v} (t_{\infty}, x) = \verb"S"(t_{\infty}) \b{\rm v}_0 $ and solve the initial value problem $ \b{\rm u} \big\vert_{t=t_{\infty}} = \b{\rm u}_0 $ in any time interval $ I \ni t_{\infty} $ to obtain
\begin{equation*}
\begin{aligned}
& ~ \lim_{n \to \infty} \Vert \b{\rm u}(t_n, \cdot) - \b{\rm v}(t_n, \cdot) \Vert_{{\dot{\rm H}}^1(\mathbb{R}^6)} \\
\le & ~ \lim_{n \to \infty} \left\Vert \verb"S"(t_n - t_{\infty}) \b{\rm u}_0 - \verb"S"(t_n) \b{\rm v}_0 \right\Vert_{{\dot{\rm H}}^1(\mathbb{R}^6)}
+ \lim_{n \to \infty} \left\Vert \int_{t_{\infty}}^{t_n} \verb"S"(t_n - s) \b{\rm f}(\b{\rm u}(s)) {\rm d}s \right\Vert_{{\dot{\rm H}}^1(\mathbb{R}^6)} \\
= & ~ \lim_{n \to \infty} \left\Vert \verb"S"(t_n) \b{\rm v}_0 - \verb"S"(t_n) \b{\rm v}_0 \right\Vert_{{\dot{\rm H}}^1(\mathbb{R}^6)}
+ \lim_{n \to \infty} \left\Vert \int_{t_{\infty}}^{t_n} \verb"S"(t_n - s) \b{\rm f}(\b{\rm u}(s)) {\rm d}s \right\Vert_{{\dot{\rm H}}^1(\mathbb{R}^6)} \\
= & ~ 0.
\end{aligned}
\end{equation*}

{\bf Case 2.} If $ \vert t_{\infty} \vert = \infty $, we can assume $ t_{\infty} = +\infty $ without loss of generality and solve the integral equation
\begin{equation*}
\b{\rm u}(t) = \verb"S"(t) \b{\rm v}_0 + i \int_t^{\infty} \verb"S"(t-s) \b{\rm f}(\b{\rm u}(s)) {\rm d}s
\end{equation*}
in $ [t_N, +\infty) \times \mathbb{R}^6 $ for $ N \gg 1 $ so large that $ \Vert \verb"S"(t) \b{\rm v}_0 \Vert_{\mathcal{S}([t_N, +\infty), {\dot{\rm H}}^1(\mathbb{R}^6))} \le \delta_{sd} $, where $ \delta_{sd} $ is as in Theorem \ref{sd}.
Then, $ \b{\rm u} $ is a solution in $ I = [t_N, +\infty) $ with the terminal value $ \b{\rm v}_0 $.

For  large $ n $, we have
\begin{equation*}
\b{\rm u}(t_n) - \b{\rm v}(t_n)
= i \int_{t_n}^{\infty} \verb"S"(t_n-s) \b{\rm f}(\b{\rm u}(s)) {\rm d}s,
\end{equation*}
and
\begin{equation*}
\Vert \nabla \b{\rm f}(\b{\rm u}) \Vert_{\mathcal{S}^\prime((t_n, +\infty), {\rm L}^2(\mathbb{R}^6))} < +\infty,
\end{equation*}
as in the proof of Theorem \ref{sd}.
So
\begin{equation*}
\begin{aligned}
& ~ \lim_{n \to \infty} \Vert \b{\rm u}(t_n, \cdot) - \b{\rm v}(t_n, \cdot) \Vert_{{\dot{\rm H}}^1(\mathbb{R}^6)}
\le \lim_{n \to \infty} \left\Vert \int_{t_n}^{\infty} \verb"S"(t_n - s) \b{\rm f}(\b{\rm u}(s)) {\rm d}s \right\Vert_{{\dot{\rm H}}^1(\mathbb{R}^6)} \\
\le & ~ C \lim_{n \to +\infty} \left\Vert \nabla \b{\rm f}(\b{\rm u}) \right\Vert_{\mathcal{S}^\prime((t_n, \infty), {\rm L}^2(\mathbb{R}^6))}
= 0.
\end{aligned}
\end{equation*}

Combining the above two cases, we complete our proof.
\end{proof}

\begin{proposition}
There exists  at most one nonlinear profile associated with $ (\b{\rm v}_0, \{ t_n \}_{n=1}^\infty) $ on an interval $ I $ with $ t_\infty \in I $.
\end{proposition}

\begin{proof}
We can show it by contradiction. Assume $ \b{\rm u}^{(1)} $ and $ \b{\rm u}^{(2)} $ are both the nonlinear profiles mentioned.

{\bf Case 1.} If $ \vert t_{\infty} \vert < \infty $, it is obvious from
\begin{equation*}
\begin{aligned}
& ~ \lim_{n \to \infty} \left\Vert \b{\rm u}^{(1)}(t_n, \cdot) - \b{\rm u}^{(2)}(t_n, \cdot) \right\Vert_{{\dot{\rm H}}^1(\mathbb{R}^6)} \\
\le & ~ \lim_{n \to \infty} \left\Vert \b{\rm u}^{(1)}(t_n, \cdot) - \b{\rm v}(t_n, \cdot) \right\Vert_{{\dot{\rm H}}^1(\mathbb{R}^6)} + \left\Vert \b{\rm u}^{(2)}(t_n, \cdot) - \b{\rm v}(t_n, \cdot) \right\Vert_{{\dot{\rm H}}^1(\mathbb{R}^6)} \\
= & ~ 0
\end{aligned}
\end{equation*}
and the well-posedness, the uniqueness in Duhamel formula, of $ \b{\rm u} $.

{\bf Case 2.} If $ \vert t_{\infty} \vert = \infty $, we can also assume $ t_{\infty} = +\infty $.
Because $ \| \nabla \b{\rm u}^{(i)} \|_{\mathcal{S}(I, {\rm L}^2 (\mathbb{R}^6))} \le \infty $ for $ i = 1, 2 $, for any $ \eps > 0 $, there  exists $ N > 0 $ such that $ \forall ~ n \ge N $
\begin{equation*}
\left\Vert \nabla \b{\rm u}^{(i)} \right\Vert_{\mathcal{S}((t_n, +\infty), {\rm L}^2 (\mathbb{R}^6))} < \eps.
\end{equation*}
By the proof of Theorem \ref{sd}, for $ M \gg N $, we have
\begin{equation*}
\sup_{t \in (t_N, t_M)} \left\Vert \nabla \b{\rm u}^{(1)}(t, \cdot) - \nabla \b{\rm u}^{(2)}(t, \cdot) \right\Vert_{{\rm L}^2(\mathbb{R}^6)} \le C \left\Vert \nabla \b{\rm u}^{(1)}(t_M, \cdot) - \nabla \b{\rm u}^{(2)}(t_M, \cdot) \right\Vert_{{\rm L}^2(\mathbb{R}^6)},
\end{equation*}
which implies that $ \b{\rm u}^{(1)} \equiv \b{\rm u}^{(2)} $ on $ (t_N, +\infty) $ and hence on $ I $.
\end{proof}

\subsection{Existence of critical solution}\label{birth}
 In order to describe the properties of critical solution, especially compactness,  we first briefly record a Palais-Smale condition modulo symmetries to describe the compactness.

\begin{proposition}[Palais-Smale condition modulo symmetries, \cite{Visan2012}]\label{PS}
Let $ \b{\rm u}_n : I_n \times \R^6 \to \C^2$ be a sequence of solutions to \eqref{NLS system} such that
\[
\limsup_{n\to \infty} \sup_{t\in I_n} H(\b{\rm u}_n(t)) = H_c
\]
and $ \{ t_n \}_{n=1}^\infty \subset I_n $ be a sequence of time such that
\begin{equation}
\lim_{n\to \infty} S_{\ge t_n} (\b{\rm u}_n) = \lim_{n\to \infty} S_{\le t_n} (\b{\rm u}_n) = \infty.
\label{S_n}
\end{equation}
Then there exists a subsequence of $ \{ \b{\rm u}_n(t_n) \}_{n=1}^\infty $  which converges in $ {\rm\dot H}^1_x(\R^6) $ modulo symmetries.
\end{proposition}

Before the proof of Theorem \ref{minH},  we introduce a notation $L(H)$ by
\[
L(H):= \sup \left\{ S_I (\b{\rm u}) ~ \Big\vert ~ \b{\rm u} : I \times \R^6 \to \C^2, ~ \sup_{t \in I} H(\b{\rm u}(t)) \le H \right\}
\] for any $ 0 \le H \le H(\b{\rm W}) $, and here the supremum is taken over all solutions $ \b{\rm u} : I \times \R^6 \to \C^2 $ to \eqref{NLS system} obeying $ \sup_{t \in I} H(\b{\rm u}(t)) \le H $.
It is easy to see that $ L : [0, H(\b{\rm W})] \to [0, \infty] $ is non-decreasing and satisfies $ L(H(\b{\rm W})) = \infty $.

 On the other hand, by the global existence and scattering result of small data, Theorem \ref{sd} and Proposition \ref{sc}, we know
\[
L(H) \le C H^2, \quad \forall ~ H < \delta_{sd},
\]
where $ \delta_{sd} > 0 $ is the threshold arose in Theorem \ref{sd}. By the long time perturbation theory and Proposition \ref{pt}, we know that $ L $ is a continuous function, which suggests there exists a unique critical kinetic energy $ H_c $ such that
\begin{equation}
H_c \ge \delta_{sd} > 0, \qquad \text{and} \qquad L(H) \left\{
\begin{aligned}
< \infty, \quad H < H_c, \\
= \infty, \quad H \ge H_c.
\end{aligned}
\right.
\label{H_c}
\end{equation}
In particular, if $ \b{\rm u} : I \times \R^6 \to \C^2 $ is a maximal-lifespan solution to \eqref{NLS system} such that $ \sup_{t \in I} H(\b{\rm u}(t)) < H_c $, then $ \b{\rm u} $ is global and
\[
S_{\R}(\b{\rm u}) \le L \left( \sup_{t \in \R}H(\b{\rm u}(t)) \right) < \infty.
\]
Thus, the failure of Theorem \ref{classification} is equivalent to $ 0 < H_c < H(\b{\rm W}) $.

\begin{proof}[Proof of Theorem \ref{minH}]
Because of the failure of Theorem \ref{classification}, there exists a unique critical kinetic energy $ H_c $ satisfying \eqref{H_c}.
And we can find a sequence $ \b{\rm u}_n : I_n \times \R^6 \to \C^2 $ of solutions to \eqref{NLS system} with $ I_n $ compact, such that
\begin{equation}
\sup_{n \in \N} \sup_{t \in I_n} H(\b{\rm u}_n (t)) = H_c, \quad \text{and} \quad \lim_{n \to \infty} S_{I_n}(\b{\rm u}_n) = \infty.
\label{u=H_c}
\end{equation}
Without loss of generality, we can take all $ t_n \equiv 0 $ in \eqref{S_n}.
Otherwise, if $ t_n \in I_n $ be such that $ S_{\ge t_n}(\b{\rm u}_n) = S_{\le t_n}(\b{\rm u}_n) $, then let $ \b{\rm u}_n(t) := \b{\rm u}(t_n - t) $.

By Proposition \ref{PS} and passing to a subsequence if necessary,
we can find $ \{ \lambda_n \}_{n=1}^\infty \subset (0, +\infty) $ and a function $ \b{\rm u}_{c, 0} \in {\rm\dot H}^1(\R^6) $ such that
\[
\lim_{n \to \infty} \Vert \mathcal{T}_{\lambda_n} \b{\rm u}_n (0) - \b{\rm u}_{c, 0} \Vert_{{\rm\dot H}^1(\R^6)} = 0.
\]
Let $ \b{\rm u}_c : I_c \times \R^6 \to \C^2 $ be the maximal-lifespan solution to \eqref{NLS system} with initial data $ \b{\rm u}_c(0) = \b{\rm u}_{c, 0} $.
Proposition \ref{pt} shows that $ I_c \subset \liminf I_n $ and
\[
\lim_{n \to \infty} \Vert \mathcal{T}_{\lambda_n} \b{\rm u}_n - \b{\rm u}_c \Vert_{L_t^\infty(K, {\rm\dot H}_x^1(\R^6))} = 0, \quad \text{for all compact} ~ K \subset I_c.
\]
Thus by \eqref{u=H_c},
\begin{equation}
\sup_{t \in I} H(\b{\rm u}(t)) \le H_c.
\label{<H_c}
\end{equation}

We claim $ \b{\rm u}_c $ blows up both forward and backward in time.
Indeed, if $ \b{\rm u}_c $ does not blow up forward in time, then $ [0, \infty) \subset I_c $ and $ S_{\ge 0} (\b{\rm u}_c) < \infty $.
By Proposition \ref{pt}, this implies $ S_{\ge 0} (\b{\rm u}_n) = S_{\ge 0} (\mathcal{T}_{\lambda_n} \b{\rm u}_n) < \infty $ for sufficiently large $ n $, which is a contradiction with \eqref{u=H_c}.
A similar argument proves that $ \b{\rm u}_c $ blows up backward in time, which completes the proof of our claim.
Therefore, by \eqref{H_c},
\begin{equation}
\sup_{t \in I_c} H(\b{\rm u}_c(t)) \ge H_c.
\label{>H_c}
\end{equation}
Combining \eqref{>H_c} with \eqref{<H_c}, we can obtain
\begin{equation}
\sup_{t \in I_c} H(\b{\rm u}_c(t)) = H_c.
\label{=H_c}
\end{equation}

It remains to show that $ \b{\rm u}_c $ is almost periodic modulo symmetries, which reminds us of Proposition \ref{PS}.
Considering  an arbitrary sequence $ \{ \tau_n \}_{n=1}^\infty \subset I_c $,
thanks to the fact that $ \b{\rm u}_c $ blows up in both time directions, we have
\[
S_{\ge \tau_n} (\b{\rm u}_c) = S_{\le \tau_n} (\b{\rm u}_c) = \infty.
\]
By Proposition \ref{PS}, there exists a subsequence of $ \b{\rm u}_c (\tau_n) $, which is convergent in $ {\rm\dot H}_x^1 (\R^6) $ modulo symmetries.
This implies the orbit $ \mathcal{F}_c := \left\{ \mathcal{T}_{\lambda_n} \b{\rm u}_c(\tau_n) ~ \vert ~ \tau_n \in I_c \right\} $ is precompact in $ {\rm\dot H}_x^1 (\R^6) $ modulo symmetries.
By the precompacness of $ \mathcal{F}_c $ in $ {\dot H}^1(\R^6) $ and the definition of $ \mathcal{T}_{\lambda} $ in Theorem \ref{lpd},
\[
\mathcal{T}_{\lambda_n} \b{\rm u}_c(\tau_n, x) = \left( \frac{1}{(\lambda_n)^2} u_c \left( \tau_n, \frac{x - x(\tau_n)}{\lambda_n} \right), \frac{1}{(\lambda_n)^2} v_c \left( \tau_n, \frac{x - x(\tau_n)}{\lambda_n} \right) \right)^t,
\]
Recall Remark \ref{compactness in H^1}, we know for any $ \eta > 0 $, there exists a compactness modulus function $ C(\eta) > 0 $ such that
\[
\int_{\vert x - x(\tau_n) \vert \ge \frac{C(\eta)}{\lambda_n}} \left( \vert \nabla u_c(\tau_n, x) \vert^2 + \frac{\kappa}{2} \vert \nabla v_c(\tau_n, x) \vert^2 \right) {\rm d}x \le \eta, \qquad \forall ~ \tau_n \in I_c
\]
and
\[
\int_{\vert \xi \vert \ge C(\eta)\lambda_n} \left( \vert \xi \vert^2 \vert \widehat{u_c}(\tau_n, \xi) \vert^2 + \frac{\kappa}{2} \vert \xi \vert^2 \vert \widehat{v_c}(\tau_n, \xi) \vert^2 \right) {\rm d}\xi \le \eta, \qquad \forall ~ \tau_n \in I_c.
\]
Choose the frequency scale function $ \lambda(\tau_n) := \lambda_n $ in Definition \ref{almost periodic} and we complete the proof of Theorem \ref{minH}.
\end{proof}

By Remark \ref{compactness in H^1} and Theorem \ref{minH}, we have established the compactness of $ \{ \b{\rm u}_c \} \subset {\rm\dot H}^1(\R^6) $.
Therefore, we can describe them in a different way.

\begin{remark}
Similarly, one can use the strategy in \cite{Killip2013}  to  prove  Theorem \ref{enemies}.
Define $ {\rm osc}(T) := \inf_{t_0 \in J} \frac{\sup \{ \lambda(t) ~ \vert ~ t \in J ~ \text{and} ~ \vert t - t_0 \vert \le T [\lambda(t_0)]^{-2} \}}{\inf \{ \lambda(t) ~ \vert ~ t \in J ~ \text{and} ~ \vert t - t_0 \vert \le T [\lambda(t_0)]^{-2} \}} $ for any $ T > 0 $ and $ a(t_0) := \frac{\lambda(t_0)}{\sup \{ \lambda(t) ~ \vert ~ t \in J ~ \text{and} ~ t \le t_0 \}} + \frac{\lambda(t_0)}{\sup \{ \lambda(t) ~ \vert ~ t \in J ~ \text{and} ~ t \ge t_0 \}} $ for any $ t_0 \in J $.
Then the case of
\[
\lim_{T \to \infty} {\rm ost}(T) = \infty, ~ \inf_{t_0 \in J} a(t_0) = 0, ~ I_c \ne \R;
\]
and
\[
\lim_{T \to \infty} {\rm ost}(T) < \infty;
\]
and
\[
\lim_{T \to \infty} {\rm ost}(T) = \infty, ~ \inf_{t_0 \in J} a(t_0) = 0, ~ I_c = \R
\]
or
\[
\lim_{T \to \infty} osc(T) = \infty, ~ \inf_{t_0 \in J} a(t_0) > 0
\]
corresponds to the case of blowing-up solution in finite time, soliton-like solution, and low-to-high frequency cascade respectively.
\end{remark}

\section{Nonexistence of blowing-up solution in finite time}\label{nonradial}
In the remaining part, we show the scattering result in Theorem \ref{vs}.
First of all, we prove the nonexistence of blowing-up solution.

\begin{theorem}\label{death1}
There dose not exist a solution to \eqref{NLS system} which blows up in finite time.
\end{theorem}

\begin{proof}
 We argue by contradiction and suppose $ \b{\rm u}_c : I_c \times \R \to \C^2 $ is a finite-time blowing-up solution to \eqref{NLS system} with the maximal interval of existence $ I_c = (-T_-(\b{\rm u}_c), T_+(\b{\rm u}_c)) $. Without loss of generality, we may assume $ \b{\rm u}_c $ blows up forward in time, i.e. $ T_+(\b{\rm u}_c) < \infty $.

First of all, we claim that
\begin{equation}
\liminf_{t \to T_+(\b{\rm u}_c)} \lambda(t) = \infty.
\label{lam=infty}
\end{equation}
Otherwise, if $ \liminf_{t \to T_+(\b{\rm u}_c)} \lambda(t) < \infty $, we can choose $ \{ t_n \}_{n=1}^\infty \subset I_c $ such that $ t_n \to T_+(\b{\rm u}_c) $ and define $ \b{\rm u}_n : I_n \times \R^6 \to \C^2 $ such that
\[
\b{\rm u}_n (t, x) := \frac1{[\lambda(t_n)]^2} \b{\rm u}_c \left( t_n + \frac{t}{[\lambda(t_n)]^2}, x(t_n) + \frac{x}{\lambda(t_n)} \right),
\]
where $ I_n := \left\{ t ~ \big\vert ~ t_n \in I_c, t_n + \frac{t}{[\lambda(t_n)]^2} \in I_c \right\} $.
It is easy to see that $ 0 \in I_n $ and $ \{ \b{\rm u}_n (t, x) \}_{n=1}^\infty $ is a sequence of solutions to \eqref{NLS system}.
Theorem \ref{minH} tells us the fact that $ \b{\rm u}_c $ is almost periodic modulo symmetries,  combining with Remark \ref{compactness in H^1}, which  implies $ \{ \b{\rm u}_n \}_{n=1}^\infty \subset {\rm\dot H}^1(\R^6) $ is precompact.
Thus, after passing to a subsequence, there exists $ \b{\rm u}_0 $ such that
\[
\lim_{n \to \infty} \Vert \b{\rm u}_n(0) - \b{\rm u}_0 \Vert_{{\rm\dot H}_x^1(\R^6)} = 0.
\]
 We can get $ \b{\rm u}_0 \not\equiv \b{\rm 0} $, by Sobolev embedding and the conservation of energy, from $ H(\b{\rm u}_n(0)) = H(\b{\rm u}_c(t_n)) $ and $ \b{\rm u}_c \not\equiv \b{\rm 0} $.
Let $ \b{\rm u}(t, x) $ be the solution to \eqref{NLS system} associated to initial data $ \b{\rm u}_0(x) $.
On one hand, by Theorem \ref{sd} and Proposition \ref{pt}, we can take $n$ so large that there exists a compact interval $ J \subset I_c $, on which $ \b{\rm u}_n $ is well-posed with finite scattering size.
So, $ \b{\rm u}_c $ is well-posed with finite scattering size on the interval $ \left\{ t_n + \frac{t}{[\lambda(t_n)]^2} ~ \big\vert ~ t \in J \right\} $.
On the other hand, we know that $ t_n \to T_+(\b{\rm u}_c) $ and $ \liminf_{n \to \infty} \lambda(t_n) = \liminf_{t \to T_+(\b{\rm u}_c)} \lambda(t) < \infty $, combining with $ T_+(\b{\rm u}_c) < \infty $, which means $ \b{\rm u}_c $ is well-posed with finite scattering size on the interval $ \left\{ t_n + \frac{t}{[\lambda(t_n)]^2} ~ \big\vert ~ t \in J \right\} \cap (T_+(\b{\rm u}_c), +\infty) $, which is nonempty.
That is impossible since $ \b{\rm u}_c $ blows up forward in time, i.e., $ \exists ~ t_1 \in I_c $ such that $ S_{[t_1, T_+(\b{\rm u}_c))} (\b{\rm u}_c) = \infty $.
As a consequence, \eqref{lam=infty} holds.

Set $ \b{\rm u}_c := (u_c, v_c)^t $.
For any $ 0 < \eta < 1 $ and $ t \in I_c $, by H\"older's inequality, we have
\begin{equation*}
\begin{aligned}
\int_{\vert x \vert < R} \vert u_c (t, x) \vert^2 {\rm d}x
\le & ~ \int_{\vert x - x(t) \vert \le \eta R} \vert u_c(t, x) \vert^2 {\rm d}x + \int_{\vert x \vert \le R, \vert x - x(t) \vert > \eta R} \vert u_c(t, x) \vert^2 {\rm d}x \\
\le & ~ \omega(6) \eta^2 R^2 \Vert u_c \Vert_{L_x^3(\R^6)}^2 + \omega(6) R^2 \left( \int_{\vert x - x(t) \vert > \eta R} \vert u_c(t, x) \vert^3 {\rm d}x \right)^{\frac23},
\end{aligned}
\end{equation*}
where $ \omega(6) $ is a constant standing for the volume of unit ball in $ \R^6 $.
$ \forall ~ \varepsilon > 0 $, let $ \zeta = \zeta(\varepsilon, R) = 2^{-\frac94} \varepsilon^{\frac32} R^{-3} > 0 $.
Thanks to \eqref{lam=infty}, there exists $ \eta = \eta(\varepsilon, R, \zeta) $ such that $ 0 < \eta \le 2^{-\frac34} [H(\b{\rm W})]^{-\frac12} \varepsilon^{\frac12} R^{-1} $ and $ C(\eta) \ge \lambda(t) C(\zeta) $.
Thus,
\[
\eta^2 R^2 \Vert u_c \Vert_{L_x^3(\R^6)}^2 \le \eta^2 R^2 \Vert u_c \Vert_{{\dot H}_x^1(\R^6)}^2 < \eta^2 R^2 H(\b{\rm W}) \le \frac{\varepsilon}{2}
\]
and
\begin{equation*}
\begin{aligned}
R^2 \left( \int_{\vert x - x(t) \vert > \eta R} \vert u_c(t, x) \vert^3 {\rm d}x \right)^{\frac23}
\le & ~ R^2 \left( \int_{\vert x - x(t) \vert > C(\zeta) \frac{\eta R \lambda(t)}{C(\eta)}} \vert \nabla u_c(t, x) \vert^2 {\rm d}x \right)^{\frac23} \\
< & ~ R^2 \zeta^{\frac23} = \frac{\varepsilon}{2},
\end{aligned}
\end{equation*}
which implies
\begin{equation}
\limsup_{t \to \sup I_c} \int_{\vert x \vert < R} \vert u_c (t, x) \vert^2 {\rm d}x = 0, \quad \forall ~ R > 0.
\label{M=0}
\end{equation}
A similar way can help us to obtain
\begin{equation}
\limsup_{t \to \sup I_c} \int_{\vert x \vert < R} \vert v_c (t, x) \vert^2 {\rm d}x = 0, \quad \forall ~ R > 0.
\label{M=00}
\end{equation}

Let $ a(x) $ be a radial smooth function in Lemma \ref{virial} satisfying
\[
a(x) = \left\{
\begin{aligned}
1, \qquad & \vert x \vert \le R, \\
0, \qquad & \vert x \vert \ge 2R,
\end{aligned}
\right.
\]
and $ \vert \nabla a(x) \vert \lesssim \frac{1}{\vert x \vert}$.
Let
\[
V_R (t) = \int_{\mathbb{R}^6} \left( 2 \kappa \vert u \vert^2 + \vert v \vert^2 \right) a(x) {\rm d}x.
\]
One one hand, combining \eqref{M=0} with \eqref{M=00}, we know
\begin{equation}
\limsup_{t \to \sup I_c} V_R (t) = \limsup_{t \to \sup I_c} \int_{\mathbb{R}^6} \left( 2 \kappa \vert u_c (t, x) \vert^2 + \vert v_c (t, x) \vert^2 \right) a(x) {\rm d}x = 0.
\label{M=01}
\end{equation}
On the other hand, by Hardy's inequality and \eqref{u_c}, we get
\[
\begin{aligned}
\vert V_R^\prime (t) \vert & ~ = 2 \kappa \left\vert \int_{\R^6} (2\overline{u}_c\nabla u_c + \overline{v}_c\nabla v_c) \cdot \nabla a(x) {\rm d}x \right\vert \\
& ~ \lesssim_\kappa\Vert \nabla \b{\rm u}_c \Vert_{{\rm L}^2(\R^6)} \left\Vert \frac{\b{\rm u}_c}{\vert x \vert} \right\Vert_{{\rm L}^2(\R^6)} \\
& ~ \lesssim_\kappa[H(\b{\rm u}_c)]^2 < [H(\b{\rm W})]^2.
\end{aligned}
\]
So, Newton-Leibniz's formula tells us,
\[
V_R(t_1) \lesssim_\kappa I_R(t_2) + \vert t_1 - t_2 \vert [H(\b{\rm W})]^2, \qquad \forall ~ t_1, t_2 \in I_c.
\]
Let $ t_2 \to T_+(\b{\rm u}_c) $ and using \eqref{M=01}, we obtain
\[
V_R(t_1) \lesssim_\kappa \vert \sup I_c - t_1 \vert [H(\b{\rm W})]^2, \qquad \forall ~ t_1 \in I_c.
\]
By the conservation of mass and the above estimate, we have
\begin{equation*}
\begin{aligned}
M(\b{\rm u}_{c, 0}) = & ~ M(\b{\rm u}_c(t_1)) \sim \int_{\R^6} 2 \kappa \vert u_c (t_1, x) \vert^2 + \vert v_c (t_1, x) \vert^2 {\rm d}x \\
= & ~ \lim_{\R \to \infty} V_{R}(t_1) \lesssim_\kappa \vert \sup I_c - t_1 \vert [H(\b{\rm W})]^2, \qquad \forall ~ t_1 \in I_c.
\end{aligned}
\end{equation*}
Let $ t_1 \to T_+(\b{\rm u}_c) $, we know $ \b{\rm u}_{c, 0} \equiv \b{\rm 0} $.
And thus, by the uniqueness of the solution to \eqref{NLS system}, $ \b{\rm u}_c \equiv \b{\rm 0} $, which is a contradiction with \eqref{u_c}.
\end{proof}

\section{Nonexistence of global blowing-up solution}\label{non}
Under the assumption of Theorem \ref{classification}, we can prove $ H_c = 0 $ for any critical solution $ \b{\rm u}_c $ to \eqref{NLS system} and show the nonexistence of the ones.
We prove the negative regularity, zero momentum and state the control of spatial center function $ x(t) $ of soliton-like solutions in Subsection \ref{properties}, which are very powerful in  excluding of our enemies.
According to the unique properties of soliton-like solution and low-to-high frequency cascade, we make full use of Theorem \ref{negative regularity} and Corollary \ref{x(t)} to rule out them one by one in Subsection \ref{D1} and Subsection \ref{D2}.

\subsection{Negative regularity}\label{properties}
We show the proof of Theorem \ref{negative regularity} in this subsection. To begin with, we present a special property of critical solution $ \b{\rm u}_c $.

\begin{proposition}[Reduced Duhamel formulas, \cite{Killip2013, Merle1998}]\label{reduced}
Let $ \b{\rm u}_c : I_c \times \R^6 \to \C^2 $ be a maximal-lifespan almost periodic modulo symmertries solution to \eqref{NLS system}.
Then $ \mathcal{S}(-t) \b{\rm u}_c $ converges weakly to $ \b{\rm 0} $ in $ {\rm\dot H}_x^1 $ as $ t \to \sup I_c $ or $ t \to \inf I_c $.
In particular, we have the `reduced' Duhamel formulas
\begin{equation*}
\begin{aligned}
\b{\rm u}_c(t) & ~ = i \lim_{T \to \sup I_c} \int_t^T \verb"S"(t-s) \b{\rm f}(\b{\rm u}_c(s)) {\rm d}s \\
& ~ = -i \lim_{T \to \inf I_c} \int_T^t \verb"S"(t-s) \b{\rm f}(\b{\rm u}_c(s)) {\rm d}s,
\end{aligned}
\end{equation*}
where the limits are to be understood in the weak $ {\rm\dot H}_x^1 $ topology.
\end{proposition}

Now, we turn to show the negative regularity of global almost periodic solution $ \b{\rm u}_c $ to \eqref{NLS system}, which plays an important role in the preclusion of soliton-like solutions and low-to-high frequency cascades.

\begin{proof}[{\bf Proof of Theorem \ref{negative regularity}.}]
  Since $ \b{\rm u}_c(t) $ is  almost periodic modulo symmetries with $ \lambda(t) \geq 1 $, then there exists $ N_0 = N_0(\eta) $ such that
\begin{align}\label{Small-compact}
\|\nabla (P_{\leq N_0} \b{\rm u}_c)\|_{{\rm L}_t^\infty(\R, {\rm L}_x^2(\R^6))} \leq \eta, \qquad \forall ~ \eta > 0.
\end{align}
Denote $ A(N) $ by
$$ A(N) := N^{-\frac12} \| P_{N} \b{\rm u}_c \|_{{\rm L}_t^\infty(\R, {\rm L}_x^4(\R^6))}, $$
for $ N \leq 10 N_0 $.

\begin{lemma}[Recurrence]\label{recurrence}
Let $A(N)$ be defined as above, for any $ N \leq 10 N_0 $, we have
\begin{equation}
A(N) \lesssim_{\b{\rm u}_c} \left(\frac{N}{N_0}\right)^{\frac12}+\eta\sum_{\frac{N}{10}\leq N_1\leq N_0}\left(\frac{N}{N_1}\right)^{\frac12}A(N_1)+\eta\sum_{N_1<\frac{N}{10}}\left(\frac{N_1}{N}\right)^{\frac12}A(N_1).
\label{AN}
\end{equation}
\end{lemma}

\begin{proof}
Fix $ N \le 10 N_0 $. By time-translation symmetry, it suffices to prove
\begin{align}\nonumber
N^{-\frac12} \| \b{\rm u}_{c, 0} \|_{{\rm L}^{4}} \lesssim \left(\frac{N}{N_0}\right)^{\frac12} + & ~ \eta \sum_{\frac{N}{10} \le N_1 \le N_0} \left( \frac{N}{N_1} \right)^{\frac12} A(N_1) \\
+ & ~ \eta \sum_{N_1 < \frac{N}{10}} \left( \frac{N_1}{N} \right)^{\frac12} A(N_1).
\end{align}

By Lemma \ref{Bern}, Lemma \ref{disper}, and Proposition \ref{reduced}, we get
\begin{equation*}
\begin{aligned}
N^{-\frac12} \| P_N \b{\rm u}_{c, 0} \|_{{\rm L}^{4}(\R^6)}
\leq & ~ N^{-\frac12} \Big\| \int_{0}^{N^{-2}} \verb"S"(- \tau) P_N \b{\rm f}(\b{\rm u}_c(\tau)) {\rm d}\tau \Big\|_{{\rm L}_x^4(\R^6)} \\
& ~ + N^{-\frac12} \int_{N^{-2}}^{\infty} \big\| \verb"S"(- \tau) P_N\b{\rm f}(\b{\rm u}_c(\tau)) \big\|_{{\rm L}_x^4(\R^6)} {\rm d}\tau \\
\lesssim & ~ N^{\frac12} \| P_N \b{\rm f}(\b{\rm u}_c) \|_{{\rm L}_t^\infty(\R, {\rm L}_x^\frac43(\R^6))}.
\end{aligned}
\end{equation*}

Let $ \b{\rm f}(\b{\rm z}) = (f_1(\b{\rm z}), f_2(\b{\rm z}))^T $ for any $ \b{\rm z} = (z_1, z_2)^T \in \C^2 $, where $ f_1(\b{\rm z}) = z_2 \bar{z}_1 $ and $ f_2(\b{\rm z}) = z_1^2 $.
Thus, for $ \b{\rm u}_c = (u_c, v_c)^T $, we have
\begin{align*}
f_1(\b{\rm u}_c) - f_1(P_{>N_0} \b{\rm u}_c)
= & ~ v_c \overline{u_c} - (P_{>N_0} v_c) (P_{>N_0} \overline{u_c}) \\
= & ~ (P_{\le N} v_c + P_{>N} v_c) (P_{\le N} \overline{u_c} + P_{>N} \overline{u_c}) - (P_{>N} v_c) (P_{>N} \overline{u_c}) \\
= & ~ (P_{\le N} v_c) (P_{\le N} \overline{u_c}) + (P_{> N} v_c) (P_{\le N} \overline{u_c}) + (P_{\le N} v_c) (P_{> N} \overline{u_c}),
\end{align*}
and
\begin{align*}
f_2(\b{\rm u}_c) - f_2(P_{>N_0} \b{\rm u}_c)
= & ~ u_c^2 + (P_{>N_0} u_c)^2 \\
= & ~ (P_{\le N} u_c + P_{>N} u_c)^2 - (P_{>N} u_c)^2 \\
= & ~ (P_{\le N} u_c)^2 + 2 (P_{> N} u_c) (P_{\le N} u_c).
\end{align*}
By the Fundamental Theorem of Calculus, we obtain
\begin{align*}
f_i(\b{\rm z}^1) - f_i(\b{\rm z}^2)
= & ~ (\b{\rm z}^1-\b{\rm z}^2) \int_{0}^{1} \partial_{\b{\rm z}} f_i(\b{\rm z}^1 + \theta (\b{\rm z}^1-\b{\rm z}^2)) {\rm d}\theta \\
& ~ + \overline{(\b{\rm z}^1-\b{\rm z}^2)} \int_{0}^{1} \partial_{\bar{\b{\rm z}}} f_i(\b{\rm z}^1 + \theta (\b{\rm z}^1-\b{\rm z}^2)) {\rm d}\theta, \quad i=1,2.
\end{align*}
where $ \partial_{\b{\rm z}} f_1(\b{\rm z}) = (0, \bar{z}_1)^T $, $ \partial_{\bar{\b{\rm z}}} f_1(\b{\rm z}) = (z_2, 0)^T $, $ \partial_{\b{\rm z}} f_2(\b{\rm z}) = (2z_1, 0)^T $ and $ \partial_{\bar{\b{\rm z}}} f_2(\b{\rm z}) = \b{\rm 0} $.

Then, we get
\begin{align*}
\b{\rm f}(\b{\rm u}_c)
= & ~ (f_1(\b{\rm u}_c), f_2(\b{\rm u}_c))^T \\
= & ~ \Big( (P_{> N_0} v_c) (P_{\le N_0} \overline{u_c}) + (P_{\le N_0} v_c) (P_{> N_0} \overline{u_c}), 2 (P_{> N_0} u_c) (P_{\le N_0} u_c) \Big)^T \\
& ~ + \b{\rm f}(P_{>N_0} \b{\rm u}_c) + \b{\rm f}(P_{\frac{N}{10} \le \cdot \le N_0} \b{\rm u}_c) \\
& ~ + \Big( P_{<\frac{N}{10}} \b{\rm u}_c \int_{0}^{1} \partial_{\b{\rm z}} f_1(P_{\frac{N}{10} \le \cdot \le N_0} \b{\rm u}_c + \theta P_{<\frac{N}{10}} \b{\rm u}_c){\rm d}\theta \\
& ~ \quad + \overline{P_{<\frac{N}{10}} \b{\rm u}_c} \int_{0}^{1} \partial_{\bar{\b{\rm z}}} f_1(P_{\frac{N}{10} \le \cdot \le N_0} \b{\rm u}_c + \theta P_{<\frac{N}{10}} \b{\rm u}_c){\rm d}\theta, \\
& ~ \qquad P_{<\frac{N}{10}} \b{\rm u}_c \int_{0}^{1} \partial_{\b{\rm z}} f_2(P_{\frac{N}{10} \le \cdot \le N_0} \b{\rm u}_c + \theta P_{<\frac{N}{10}} \b{\rm u}_c) {\rm d}\theta \\
& ~ \qquad + \overline{P_{<\frac{N}{10}} \b{\rm u}_c} \int_{0}^{1} \partial_{\bar{\b{\rm z}}} f_2(P_{\frac{N}{10} \le \cdot \le N_0} \b{\rm u}_c + \theta P_{<\frac{N}{10}} \b{\rm u}_c){\rm d}\theta\Big)^T \\
=: & ~ \mathcal{A + B + C + D}.
\end{align*}
Using H\"older and Bernstein,
\begin{align*}
N^{\frac12} \left\Vert P_N (\mathcal{A + B}) \right\Vert_{{\rm L}_t^\infty(\R, {\rm L}_x^{\frac43}(\R^6))}
\lesssim & ~ N^{\frac12} \| P_{>N_0} \b{\rm u}_c \|_{{\rm L}_t^\infty(\R, {\rm L}_x^{\frac{12}5}(\R^6))} \| \b{\rm u}_c \|_{{\rm L}_t^\infty(\R, {\rm L}_x^{3}(\R^6))} \\
\lesssim & ~ N^{\frac12} N_0^{\frac12}.
\end{align*}

Next, we turn to estimate the fourth term $ \mathcal{D} $ and define $ \mathcal{D} = \mathcal{D}_1 + \mathcal{D}_2 $ for
\begin{equation*}
\begin{aligned}
\mathcal{D}_1 = & ~ \Big( P_{<\frac{N}{10}} \b{\rm u}_c \int_{0}^{1} \partial_{\b{\rm z}} f_1(P_{\frac{N}{10} \le \cdot \le N_0} \b{\rm u}_c + \theta P_{<\frac{N}{10}} \b{\rm u}_c){\rm d}\theta, \\
& ~ \quad + P_{<\frac{N}{10}} \b{\rm u}_c \int_{0}^{1} \partial_{\b{\rm z}} f_2(P_{\frac{N}{10} \le \cdot \le N_0} \b{\rm u}_c + \theta P_{<\frac{N}{10}} \b{\rm u}_c) {\rm d}\theta \Big)^T,
\end{aligned}
\end{equation*}
and
\begin{equation*}
\begin{aligned}
\mathcal{D}_2 = & ~ \Big( \overline{P_{<\frac{N}{10}} \b{\rm u}_c} \int_{0}^{1} \partial_{\bar{\b{\rm z}}} f_1(P_{\frac{N}{10} \le \cdot \le N_0} \b{\rm u}_c + \theta P_{<\frac{N}{10}} \b{\rm u}_c){\rm d}\theta, \\
& ~ \quad + \overline{P_{<\frac{N}{10}} \b{\rm u}_c} \int_{0}^{1} \partial_{\bar{\b{\rm z}}} f_2(P_{\frac{N}{10} \le \cdot \le N_0} \b{\rm u}_c + \theta P_{<\frac{N}{10}} \b{\rm u}_c){\rm d}\theta\Big)^T. \\
\end{aligned}
\end{equation*}
It suffices to consider $ \mathcal{D}_1 $  since $ \mathcal{D}_2 $ can be handled similarly.
Bernstein inequality yields
\[
\| P_{>\frac{N}{10}} \partial_{\b{\rm z}} f_i(\b{\rm u}_c) \|_{{\rm L}_t^\infty(\R, {\rm L}_x^2(\R^6))} \lesssim N^{-1} \| \nabla \b{\rm u}_c \|_{{\rm L}_t^\infty(\R, {\rm L}_x^2(\R^6))}, \qquad i = 1, 2.
\]
Thus, by H\"older's inequality  and \eqref{Small-compact},
\begin{align*}
& ~ N^{\frac12} \left\Vert P_N \mathcal{D}_1 \right\Vert_{{\rm L}_t^\infty(\R, {\rm L}_x^{\frac43}(\R^6))} \\
= & ~ N^{\frac12} \sum_{i=1}^2\Big\| P_N \Big( P_{<\frac{N}{10}} \b{\rm u}_c \int_{0}^{1} \partial_{\b{\rm z}} f_i(P_{\frac{N}{10} \le \cdot \le N_0} \b{\rm u}_c + \theta P_{<\frac{N}{10}} \b{\rm u}_c) {\rm d}\theta \Big) \Big\|_{{\rm L}_t^\infty(\R, {\rm L}_x^{\frac43}(\R^6))} \\
\lesssim & ~ N^{\frac12} \| P_{< \frac{N}{10}} \b{\rm u}_c \|_{{\rm L}_t^\infty(\R, {\rm L}_x^4(\R^6))}
\sum_{i=1}^2 \Big\| P_{> \frac{N}{10}} \Big( \int_{0}^{1} \partial_{\b{\rm z}} f_i(P_{\frac{N}{10} \le \cdot \le N_0} \b{\rm u}_c + \theta P_{<\frac{N}{10}} \b{\rm u}_c){\rm d}\theta \Big) \Big\|_{{\rm L}_t^\infty(\R, {\rm L}_x^2(\R^6))} \\
\lesssim & ~ N^{-\frac12} \| P_{< \frac{N}{10}} \b{\rm u}_c \|_{{\rm L}_t^\infty(\R, {\rm L}_x^4(\R^6))} \| \nabla P_{\le N_0} \b{\rm u}_c \|_{{\rm L}_t^\infty(\R, {\rm L}_x^2(\R^6))} \\
\lesssim & ~ \eta \sum_{N_1 < \frac{N}{10}} \Big( \frac{N_1}{N} \Big)^{\frac12} A(N_1).
\end{align*}
Then, we have
\begin{equation*}
N^{\frac12} \| P_{N}\mathcal{D} \|_{{\rm L}_t^\infty(\R, L_x^{\frac43}(\R^6))}
= N^{\frac12} \sum_{j=1}^2 \| P_{N}\mathcal{D}_j \|_{{\rm L}_t^\infty(\R, L_x^{\frac43}(\R^6))}
\lesssim \eta \sum_{N_1<\frac{N}{10}} \Big( \frac{N_1}{N} \Big)^{\frac12} A(N_1).
\end{equation*}

We are left to estimate the contribution of $ \mathcal{C} $. It remains to show
\[
\| P_N\mathcal{C} \|_{{\rm L}_t^\infty(\R, {\rm L}_x^{\frac43}(\R^6))} \lesssim \eta \sum_{\frac{N}{10} \leq N_1 \leq N_0} N_1^{-\frac12} A(N_1).
\]
Using the triangle inequality, Bernstein, \eqref{Small-compact} and H\"older, we  estimate as follow:
\begin{align*}
& ~ \| P_N\mathcal{C} \|_{{\rm L}_t^\infty(\R, {\rm L}_x^{\frac43}(\R^6))}
\lesssim \| \b{\rm f} (P_{\frac{N}{10} \le \cdot \le N_0} \b{\rm u}_c) \|_{{\rm L}_t^\infty(\R, L_x^{\frac43}(\R^6))} \\
\lesssim & ~ \sum_{\frac{N}{10} \le N_1, N_2 \le N_0} \Big( \| \overline{(P_{N_1} u_c)} (P_{N_2} v_c) \|_{L_t^\infty(\R, L_x^{\frac43}(\R^6))} + \| (P_{N_1} u_c) (P_{N_2} u_c) \|_{L_t^\infty(\R, L_x^{\frac43}(\R^6))} \Big) \\
\lesssim & ~ \eta \sum_{\frac{N}{10} \le N_1 \le N_2 \le N_0} N_2^{-1} \| P_{N_1} u_c \|_{L_t^\infty(\R, L_x^4(\R^6))} \\
& ~ + \sum_{\frac{N}{10} \le N_2 \le N_1 \le N_0} \| P_{N_1} u_c \|_{L_t^\infty(\R, L_x^2(\R^6))} \| P_{N_2} v_c \|_{L_t^\infty(\R, {\rm L}_x^4(\R^6))} \\
\lesssim & ~ \eta \sum_{\frac{N}{10} \le N_1 \le N_0} N_1^{-\frac12} A(N_1) + \eta \sum_{\frac{N}{10} \le N_2 \le N_1 \le N_0} \Big( \frac{N_2}{N_1} \Big) (N_2^{-\frac12} A(N_2)) \\
\lesssim & ~ \eta \sum_{\frac{N}{10} \le N_1 \le N_0} N_1^{-\frac12} A(N_1).
\end{align*}

Finally, we put the above estimates together to obtain
\begin{equation*}
\begin{aligned}
A(N) & ~ = N^{\frac12} \| P_{N}(\mathcal{A + B + C + D}) \|_{{\rm L}_t^\infty(\R, L_x^{\frac43}(\R^6))} \\
& ~ \lesssim_{\b{\rm u}_c} \left(\frac{N}{N_0}\right)^{\frac12}+\eta\sum_{\frac{N}{10}\leq N_1\leq N_0}\left(\frac{N}{N_1}\right)^{\frac12}A(N_1)+\eta\sum_{N_1<\frac{N}{10}}\left(\frac{N_1}{N}\right)^{\frac12}A(N_1),
\end{aligned}
\end{equation*}
which implies that \eqref{AN} holds.
\end{proof}

\begin{proposition}\label{ne}
Let $ \b{\rm u}_c $ be as in Theorem \ref{negative regularity}. Then
\[
\b{\rm u}_c \in {\rm L}_t^\infty(\R, {\rm L}_x^p(\R^6)) \quad \text{for} ~ \frac{14}{5} \le p < 3.
\]
In particular, by H\"older's inequality,
\[
\nabla \b{\rm f}(\b{\rm u}_c) \in {\rm L}_t^\infty(\R, {\rm L}_x^r(\R^6)) \quad \text{for} ~ \frac76 \le r < \frac65.
\]
\end{proposition}

\begin{proof}
Combining Lemma \ref{recurrence} with Lemma \ref{Gronwall}, we deduce
\begin{equation}
\Vert P_N \b{\rm u}_c \Vert_{{\rm L}_t^\infty(\R, {\rm L}_x^{\frac43}(\R^6))} \lesssim N \quad \text{for} ~ N \leq 10 N_0.
\label{<N}
\end{equation}
In fact, setting $N=10\cdot2^{-k}N_0,$ $x_k=A(10\cdot2^{-k}N_0)$ and $\gamma=1$ in Lemma \ref{Gronwall}, we can obtain
\[
A(10\cdot2^{-k}N_0)\lesssim \sum_{l=0}^{\infty} r^{|k-l|}(10\cdot2^{-l})^\frac12.
\]
Then taking $ \eta > 0 $ small enough such that $ r < \sqrt{2} $, we have
\[
A(10\cdot2^{-k}N_0)\lesssim (10\cdot2^{-k})^\frac12\sum_{l=0}^{\infty} r^{|k-l|}(10\cdot2^{-l+k})^\frac12\lesssim (10\cdot2^{-k})^\frac12
.\]
Observing the definition of $A(N)$, we deduce \eqref{<N} by
\[
\Vert P_N \b{\rm u}_c \Vert_{{\rm L}_t^\infty(\R, {\rm L}_x^{\frac43}(\R^6))} \lesssim N_0^{-\frac12} N^{\frac12}(10\cdot2^{-k}N_0)^\frac12 \lesssim N.
\]

By interpolation, we get
\begin{equation*}
\| P_N \b{\rm u}_c \|_{{\rm L}_t^\infty(\R, {\rm L}_x^p(\R^6))}
\lesssim \| P_N \b{\rm u}_c \|_{{\rm L}_t^\infty(\R, L_x^4(\R^6))}^{\frac{2(p-2)}{p}} \| P_N \b{\rm u}_c \|_{{\rm L}_t^\infty(\R, {\rm L}_x^2(\R^6))}^{\frac4p-2}
\lesssim N^{3-\frac{8}{p}}
\lesssim N^{\frac17}
\end{equation*}
for all $ N \leq 10 N_0 $. Using Lemma \ref{Bern},
\begin{equation*}
\begin{aligned}
\| \b{\rm u}_c \|_{{\rm L}_t^\infty(\R, {\rm L}_x^p(\R^6))}
\leq & ~ \| P_{\leq N_0} \b{\rm u}_c \|_{{\rm L}_t^\infty(\R, {\rm L}_x^p(\R^6))} + \| P_{>N_0} \b{\rm u}_c \|_{{\rm L}_t^\infty(\R, {\rm L}_x^p(\R^6))} \\
\lesssim & ~ \sum_{N \leq N_0} N^{\frac17} + \sum_{N > N_0} N^{2-\frac6p}
\lesssim 1,
\end{aligned}
\end{equation*}
which completes the proof of the proposition.
\end{proof}

We can use the decomposition of frequency to deduce the following proposition.

\begin{proposition}[Some negative regularity, \cite{Killip2010}]\label{n}
Let $ \b{\rm u}_c $ be as in Theorem \ref{negative regularity}.
If $ |\nabla|^s \b{\rm f}(\b{\rm u}_c) \in {\rm L}_t^\infty(\R, {\rm L}_x^p(\R^6)) $ for some $ \frac76 \leq p < \frac65 $ and some $ s \in [0, 1] $,
then there exists $ s_0 = s_0(p) > 0 $ such that $ \b{\rm u}_c \in {\rm L}_t^\infty(\R, {\rm\dot H}_x^{s-s_0+}(\R^6)) $.
\end{proposition}

 To end the proof of Theorem \ref{negative regularity}, we first apply Proposition \ref{n} with $ s = 1 $ to show that $ \b{\rm u} \in {\rm L}_t^\infty({\rm\dot H}_x^{1-s_0+}(\R^6)) $ for some $ s_0 > 0 $.
The fractional chain rule and Proposition \ref{ne} tell us that $ |\nabla|^{1-s_0+} \b{\rm f}(\b{\rm u}_c) \in {\rm L}_t^\infty(\R, {\rm L}_x^p(\R^6)) $ for some $ \frac76 \leq p < \frac65 $.
Thus,another application of  Proposition \ref{n} facilitates us to obtain $ \b{\rm u} \in {\rm L}_t^\infty({\rm\dot H}_x^{1-2s_0+}(\R^6)) $.
Iterating this procedure finitely many times, we derive $ \b{\rm u} \in {\rm L}_t^\infty({\rm\dot H}_x^{-\eps}(\R^6)) $ for any $ 0 < \eps < s_0 $.
\end{proof}

\begin{proposition}\label{P=0}
If $ \b{\rm u}_c := (u_c, v_c)^T $ be a minimal-kinetic-energy blowing-up solution to \eqref{NLS system} satisfying $ \b{\rm u}_c \in {\rm L}_t^\infty(I_c, {\rm H}_x^1(\R^6)) $ and its initial data $ \b{\rm u}_{c, 0} $ to be non-radial with the mass-resonance condition or $ \b{\rm u}_{c, 0} $ to be radial, then its momentum is zero, that is,
\begin{equation}
P(\b{\rm u}_c) := \Im \int_{\mathbb{R}^6} \left( \overline{u}_c \nabla u_c + \frac{1}{2}\overline{v_c} \nabla v_c \right) {\rm d}x \equiv \b{\rm 0}.
\label{p=0}
\end{equation}
\end{proposition}

\begin{proof}
First of all, we know that $ \b{\rm u}_c $ will be radial if its initial data $ \b{\rm u}_{c, 0} $ is radial.
Combining with the fact that $ \nabla f(x) = f^\prime(r) \frac{x}{r} $ for radial function $ f(r) := f(x) $ when $ r = \vert x \vert $, we have
\begin{equation*}
\begin{aligned}
P(\b{\rm u}_c) = & ~ \Im \int_{\mathbb{R}^6} \left( \overline{u}_c \nabla u_c + \frac{1}{2}\overline{v_c} \nabla v_c \right) {\rm d}x \\
= & ~ \Im \int_{\mathbb{R}^6} \left( \overline{u_c(r)} u_c^\prime(r) \frac{x}{r} + \frac{1}{2}\overline{v_c(r)} v_c^\prime(r) \frac{x}{r} \right) {\rm d}x \\
\equiv & ~ \b{\rm 0},
\end{aligned}
\end{equation*}
where we have used the fact that $ \int_{\R^6} g(x) {\rm d}x \equiv 0 $ for any odd function.

We are left to show the remaining case, i.e. $ \b{\rm u}_{c, 0} $ is  non-radial with the mass-resonance condition.
Under the mass-resonance condition, that is $\kappa=\frac{1}{2}$, \eqref{NLS system} enjoys the Galilean invariance property:
\[
\b{\rm u}_c(t, x)
= \left(
\begin{aligned}
u_c(t, x) \\
v_c(t, x)
\end{aligned}
\right) \to \left(
\begin{aligned}
u_c^{\xi}(t, x) \\
v_c^{\xi}(t, x)
\end{aligned}
\right)
=: \left(
\begin{aligned}
e^{ix\cdot\xi}e^{-it\vert\xi\vert^2}u_c(t, x-2t\xi) \\
e^{2ix\cdot\xi}e^{-2it\vert\xi\vert^2}v_c(t, x-2t\xi)
\end{aligned}
\right)
\]
for any $ \xi \in \mathbb{R}^6 $.
Then, for $ \b{\rm u}_c^{\xi} := (u_c^{\xi}, v_c^{\xi})^T $, we can compute
\[
\begin{aligned}
H(\b{\rm u}_c^{\xi}) = & ~ \Vert \nabla u_c^{\xi} \Vert_{L^2(\R^6)}^2 + \frac{\kappa}{2} \Vert \nabla v_c^{\xi} \Vert_{L^2(\R^6)}^2 \\
= & ~ \left\Vert i \xi u_c + \nabla u_c \right\Vert_{L^2(\R^6)}^2 + \frac{\kappa}{2} \left\Vert 2i \xi v_c + \nabla v_c \right\Vert_{L^2(\R^6)}^2 \\
= & ~ \int_{\R^6} \left\vert \Re(\nabla u_c) - \xi \Im u_c \right\vert^2 + \left\vert \xi \Re u_c + \Im(\nabla u_c) \right\vert^2 {\rm d}x \\
 & ~ + \frac{\kappa}{2} \int_{\R^6} \left\vert \Re(\nabla v_c) - 2\xi \Im v_c \right\vert^2 + \left\vert 2\xi \Re v_c + \Im(\nabla v_c) \right\vert^2 {\rm d}x \\
= & ~ \int_{\R^6} \left\vert \nabla u_c \right\vert^2 + \vert \xi \vert^2 \vert u_c \vert^2 + 2\xi \left( \Re u_c \Im(\nabla u_c) - \Re(\nabla u_c) \Im u_c \right) {\rm d}x \\
 & ~ + \frac{\kappa}{2} \int_{\R^6} \left\vert \nabla v_c \right\vert^2 + 4 \vert \xi \vert^2 \vert v_c \vert^2 + 4\xi \left( \Re v_c \Im(\nabla v_c) - \Re(\nabla v_c) \Im v_c \right) {\rm d}x.
\end{aligned}
\]
Note that $ \kappa = \frac12 $, we have
\[
\begin{aligned}
& ~ H(\b{\rm u}_c^{\xi}) - H(\b{\rm u}_c) - \vert \xi \vert^2 M(\b{\rm u}_c) \\
= & ~ 2\xi \cdot \int_{\R^6} \big( \Re u_c \Im(\nabla u_c) - \Re(\nabla u_c) \Im u_c \big) \\
& ~ \qquad \qquad + \frac12 \big( \Re v_c \Im(\nabla v_c) - \Re(\nabla v_c) \Im v_c \big) {\rm d}x \\
= & ~ 2\xi \cdot \int_{\R^6} \Im(\overline{u_c} \nabla u_c) + \frac12 \Im(\overline{v_c} \nabla v_c) {\rm d}x
= 2\xi \cdot P(\b{\rm u}_c).
\end{aligned}
\]

Because $ S_{I_c^{\xi}}(\b{\rm u}_c^{\xi}) = S_{I_c}(\b{\rm u}_c) = \infty $, $ \b{\rm u}_c^{\xi} $ is a blowing-up solution to \eqref{NLS system} as well.
Combining with the fact that $ \b{\rm u}_c $ has the minimal kinetic energy among all the blowing-up solutions, we obtain
\[
\vert \xi \vert^2 M(\b{\rm u}_c) + 2\xi \cdot P(\b{\rm u}_c) = H(\b{\rm u}_c^{\xi}) - H(\b{\rm u}_c) \ge 0
\]
holds for any $ \xi \in \R^6 $.
Setting $ \xi = -\frac{P(\b{\rm u}_c)}{M(\b{\rm u}_c)} $, we find
\[
0 \le \vert \xi \vert^2 M(\b{\rm u}_c) + 2\xi \cdot P(\b{\rm u}_c) = -\frac{\vert P(\b{\rm u}_c) \vert^2}{M(\b{\rm u}_c)} \le 0,
\]
which implies \eqref{p=0}.
\end{proof}

Deduced from the property of negative regularity (Theorem \ref{negative regularity}) and zero momentum (Proposition \ref{P=0}) the movement of spatial center function $ x(t) $ of soliton-like solutions can be controlled, which can help us use Lemma \ref{virial} more effectively, combined with the finite mass of soliton-like solution ($ \lambda(t) \ge 1 $).

\begin{corollary}[Control of spatial center function $ x(t) $, \cite{Killip2010}]\label{x(t)}
Let $ \b{\rm u}_c $ be a soliton-like solution as in Theorem \ref{enemies}.
Then, for any $ \eta > 0 $, there exists $ C(\eta) > 0 $ such that
\[
\sup_{t \in \R} \int_{\vert x - x(t) \vert \ge C(\eta)} \vert \b{\rm u}_c(t, x) \vert^2 {\rm d}x \lesssim \eta.
\]
Furthermore, the spatial center function $ x(t) $ satisfies
\[
\vert x(t) \vert = o(t), \qquad t \to \infty.
\]
\end{corollary}

\subsection{Exclusion of soliton-like solution}\label{D1}
Now, we  use Virial identity to rule out the soliton-like solution.
 Its compactness has been described as an almost periodic solution.
From the control of spatial center function $ x(t) $, Corollary \ref{x(t)}, we know that $ x(t) = o(t) $ which means $ \b{\rm u}_c $ moves slower and slower.
That implies the nonexistence of soliton-like solution for a large time.

\begin{theorem}\label{death2}
There exist no soliton-like solutions to \eqref{NLS system}.
\end{theorem}

\begin{proof}
We argue by contradiction and suppose $ \b{\rm u}_c : \R \times \R^6 \to \C^2 $ is a soliton-like solution to \eqref{NLS system}.
By Definition \ref{almost periodic} and $ L^3(\R^6) \hookrightarrow {\dot H}^1(\R^6) $, for any $ \eta > 0 $, there exists $ C(\eta) > 0 $ such that
\begin{equation}
\sup_{t \in \R} \int_{\vert x - x(t) \vert \ge C(\eta)} \left( \vert \nabla \b{\rm u}_c \vert^2 + \vert \b{\rm u}_c \vert^3 \right) {\rm d}x \le \eta.
\label{eta}
\end{equation}
Using Corollary \ref{x(t)}, we know there exists $ T_0 = T_0(\eta) \in \R $ such that
\begin{equation}
\vert x(t) \vert \le \eta t, \qquad \forall ~ t \ge T_0.
\label{x=0}
\end{equation}

Recalling Lemma \ref{virial}, for $ a(x) $ be a radial smooth function satisfying
\[
a(x) = \left\{
\begin{aligned}
\vert x \vert^2, \quad & \vert x \vert \le R, \\
0, \qquad & \vert x \vert \ge 2R,
\end{aligned}
\right.
\]
where $ R > 0 $ will be chosen later, we define
\[
V_R(t) = \int_{\R^6} \left( 2 \kappa \vert u_c \vert^2 + \vert v_c \vert^2 \right) a(x) {\rm d}x.
\]
By Theorem \ref{negative regularity} and interpolation theorem, we know $ \b{\rm u}_c := (u_c, v_c)^T \in {\rm L}_t^\infty(\R, {\rm L}_x^2(\R^6)) $.
Then we get
\begin{equation}
\left\vert V_R^\prime(t) \right\vert = \left\vert 2\kappa \Im \int_{\mathbb{R}^6} \left( 2 \overline{u_c} \nabla u_c + \overline{v_c} \nabla v_c \right) \cdot \nabla a(x) {\rm d}x \right\vert
\lesssim_\kappa R H(\b{\rm u}_c) M(\b{\rm u}_c) \lesssim R,
\label{Vp}
\end{equation}
and
\[
\begin{aligned}
V_R^{\prime\prime}(t) = & ~ 8\kappa \Re \int_{\mathbb{R}^6} \left( \overline{u_{c, j}} u_{c, k} + \frac{\kappa}{2} \overline{v_{c, j}} v_{c, k} \right) a_{jk}(x) {\rm d}x \\
& ~ - 2\kappa \int_{\mathbb{R}^6} \left( \vert u_c \vert^2 + \frac{\kappa}{2} \vert v_c \vert^2 \right) \Delta \Delta a(x) {\rm d}x - 2\kappa \Re \int_{\mathbb{R}^6} \overline{v_c} u_c^2 \Delta a(x) {\rm d}x \\
= & ~ 8\kappa [ 2 H(\b{\rm u}_c) - 3 R(\b{\rm u}_c) ] \\
& ~ + O\left( \int_{\vert x \vert \ge R} \vert \nabla \b{\rm u}_c \vert^2 + \vert \b{\rm u}_c \vert^3 {\rm d}x \right) + O\left( \int_{R \le \vert x \vert \le 2R} \vert \b{\rm u}_c \vert^3 {\rm d}x \right)^{\frac23}.
\end{aligned}
\]
For $ \eta > 0 $ small enough, choosing
\[
R := C(\eta) + \sup_{T_0 \le t \le T_1} \vert x(t) \vert
\]
for any $ T_1 > T_0 $, by coercivity of energy (Proposition \ref{coer}), energy trapping (Corollary \ref{Energy trapping}), and \eqref{eta}, we have
\begin{equation}
V^{\prime\prime}(t) \gtrsim_\kappa E(\b{\rm u}_{c, 0}).
\label{Vpp}
\end{equation}

Applying the Fundamental Theorem of Calculus on $ [T_0, T_1] $, by \eqref{x=0}, \eqref{Vp} and \eqref{Vpp}, we obtain
\[
(T_1 - T_0) E(\b{\rm u}_{c, 0})  \lesssim_{\kappa} R = C(\eta) + \sup_{T_0 \le t \le T_1} \vert x(t) \vert \le C(\eta) + \eta T_1, \quad \forall ~ T_1 > T_0.
\]
Setting first $ \eta \to 0 $ and then $ T_1 \to \infty $, we find $ E(\b{\rm u}_{c, 0}) = 0 $.
Using the conservation of energy and energy trapping again, we know $ H(\b{\rm u}_c(t)) = 0, ~ \forall ~ t \in \R $.
By \eqref{=H_c}, we get $ H_c = 0 $ which is a contradiction with $ H_c \ge \delta_{sd} > 0 $ in \eqref{H_c}.
\end{proof}

\subsection{Exclusion of low-to-high frequency cascade}\label{D2}

 Our last enemy is the low-to-high cascade. An important ingredient is the negative regularity lemma, which indicates that ${\rm u}_c$ also falls into the $L^2$ space.

\begin{theorem}\label{death3}
There exist no low-to-high frequency cascade to \eqref{NLS system}.
\end{theorem}

\begin{proof}
We argue by contradiction and suppose $ \b{\rm u}_c : I_c \times \R \to \C^2 $ is a low-to-high frequency cascade solution to \eqref{NLS system}.
Similar to the proof of Theorem \ref{death2}, we know $ \b{\rm u}_c \in {\rm L}_t^\infty(\R, {\rm L}_x^2(\R^6)) $.
Noticing the conservation of mass, we obtain
\[
0 \le M(\b{\rm u}_c) = M(\b{\rm u}_c(t)) := \int_{\R^6} \vert \b{\rm u}_c(t, x) \vert^2 {\rm d}x < \infty,
\]
for any $ t \in \R $.

On one hand, fixing $ t \in \R $ and choosing $ \eta > 0 $ to be a small number, by Remark \ref{compactness in H^1} and Theorem \ref{minH}, we have
\begin{equation}
\int_{\vert \xi \vert \le C(\eta) \lambda(t)} \vert \xi \vert^2 \vert \widehat{\b{\rm u}_c}(t, \xi) \vert^2 {\rm d}\xi < \eta.
\label{<H1}
\end{equation}
At the same time, since $ \b{\rm u}_c \in {\rm L}_t^\infty(\R, {\dot\rm H}_x^{-\varepsilon}(\R^6)) $, we know
\begin{equation}
\int_{\vert \xi \vert \le C(\eta) \lambda(t)} \vert \xi \vert^{-2\varepsilon} \vert \widehat{\b{\rm u}_c}(t, \xi) \vert^2 {\rm d}\xi \lesssim 1.
\label{<H-}
\end{equation}
Thus, interpolating \eqref{<H1} and \eqref{<H-} yields
\begin{equation}
\int_{\vert \xi \vert \le C(\eta) \lambda(t)} \vert \widehat{\b{\rm u}_c}(t, \xi) \vert^2 {\rm d}\xi \lesssim \eta^{\frac{\varepsilon}{1 + \varepsilon}}.
\label{<L2}
\end{equation}

On the other hand, by Theorem \ref{minH},
\begin{equation}
\begin{aligned}
\int_{\vert \xi \vert > C(\eta) \lambda(t)} \vert \widehat{\b{\rm u}_c}(t, \xi) \vert^2 {\rm d}\xi
\le & ~ [C(\eta) \lambda(t)]^{-2} \int_{\R^6} \vert \xi \vert^2 \vert \widehat{\b{\rm u}_c}(t, \xi) \vert^2 {\rm d}\xi \\
\lesssim & ~ [C(\eta) \lambda(t)]^{-2} H(\b{\rm u}_c(t)) \\
< & ~ [C(\eta) \lambda(t)]^{-2} H(\b{\rm W}).
\end{aligned}
\label{>L2}
\end{equation}

Combining \eqref{<L2}, \eqref{>L2}, and Plancheral's theorem, we get the conclusion that
\[
0 \le M(\b{\rm u}_c) \lesssim \eta^{\frac{\varepsilon}{1 + \varepsilon}} + [C(\eta) \lambda(t)]^{-2}, \qquad \forall ~ t \in \R.
\]
By Definition \ref{enemies}, there exists a time sequence $ \{ t_n \}_{n=1}^\infty \subset \R^+ $ such that
\[
\lim_{n \to \infty} t_n = +\infty, \quad \text{and} \quad \lim_{n \to \infty} \lambda(t_n) = +\infty.
\]
{\color{red} },
\[
0 \le \lim_{n \to \infty} M(\b{\rm u}_c(t_n)) \lesssim \eta^{\frac{\varepsilon}{1 + \varepsilon}}.
\]
Let $ \eta \to 0 $, we have $ M(\b{\rm u}_c(t_n)) \to 0, ~ n \to \infty $.
Finally,  by the conservation of mass, we know $ \b{\rm u}_c \equiv \b{\rm 0} $, which is a contradiction with \eqref{u_c} in Theorem \ref{enemies}.
\end{proof}

\end{document}